\documentclass[11pt]{article}

\usepackage[margin=1in]{geometry}
\usepackage{amsmath,amssymb,amsthm,mathtools}
\usepackage{algorithm}
\usepackage{algorithmic}
\usepackage{enumitem}
\usepackage{hyperref}
\usepackage{comment}
\usepackage{xcolor}

\newtheorem{theorem}{Theorem}
\newtheorem{lemma}{Lemma}
\newtheorem{assumption}{Assumption}
\newtheorem{remark}{Remark}
\newtheorem{definition}{Definition}
\newtheorem{corollary}{Corollary}

\newcommand{\Pbb}{\mathbb P}
\newcommand{\cC}{\mathcal C}

\newcommand{\norm}[1]{\left\lVert #1\right\rVert}
\newcommand{\abs}[1]{\left|#1\right|}
\newcommand{\real}{\mathbb{R}}

\newcommand{\1}{\mathbf{1}}
\newcommand{\Ebb}{\mathbb{E}}
\newcommand{\eps}{\varepsilon}
\newcommand{\Ber}{\operatorname{Ber}}
\newcommand{\Signal}{\operatorname{Signal}}
\newcommand{\argmax}{\operatorname*{arg\,max}}

\newcommand{\Ses}{\mathcal{S}_{\mathrm{es}}}
\newcommand{\ED}{\mathcal{E}_D}

\newcommand{\one}{\mathbf{1}}

\newcommand{\poly}{\operatorname{poly}}

\newcommand{\Pp}{\mathbb P}
\newcommand{\Ee}{\mathbb E}

\newcommand{\ind}{\mathbf 1}
\newcommand{\cE}{\mathcal E}

\newcommand{\cS}{\mathcal S}
\newcommand{\Sn}{\mathfrak S}
\newcommand{\G}{\mathcal G}
\newcommand{\Tin}{\widetilde T}
\newcommand{\Hin}{\widetilde H}
\newcommand{\Psiin}{\widetilde\Psi}

\newcommand{\dorb}{d_{\mathrm{orb}}}
\newcommand{\Tr}{\operatorname{Tr}}
\newcommand{\Law}{\mathcal L}

\usepackage{tikz}
\usetikzlibrary{arrows.meta,positioning,shapes.geometric}

\begin{document}

\begin{center}
{\bf{\LARGE{Near-optimal node-private community estimation \\ in polynomial-time}}} \\
%{\bf{\LARGE{Polynomial-time sampler for node-private community estimation in stochastic block models}}} \\
\vspace*{.25in}
\begin{tabular}{ccc}
{\large{Laurentiu Marchis}} & \hspace*{.5in}  & {\large{Olga Klopp}}\\
{\large{Univesrity of Cambridge}} & \hspace*{.5in}  & {\large{ESSEC Business School}}\\
{\large{\texttt{lam223@cam.ac.uk}}} & \hspace*{.5in} & {\large{\texttt{b00732676@essec.edu}}} \\
\vspace*{.25in} \\
{\large{Po-Ling Loh}} & \hspace*{.5in}  & {\large{Ilias Zadik}}\\
{\large{University of Cambridge}} & \hspace*{.5in}  & {\large{Yale University}}\\{\large{\texttt{pll28@cam.ac.uk}}} & \hspace*{.5in} & {\large{\texttt{ilias.zadik@yale.edu}}} \\
\end{tabular}

\end{center}

%\maketitle

\begin{abstract}

In this paper, we resolve an open question of Klopp \& Zadik (2026) by providing a high-probability polynomial-time, node-private algorithm which nearly matches the performance of their exponential-time node-private algorithm for exact recovery in stochastic block models. Our result involves an explicitly constructed Lipschitz surrogate for the penalized likelihood function, as well as a carefully devised accept-reject algorithm that samples community labels from the corresponding exponential mechanism in polynomial-time. We rigorously analyze the privacy, runtime, and utility of our proposed algorithm, showing that even when the number of communities $K$ grows logarithmically with the number of nodes $n$, we can achieve the minimax rates for exact recovery with the privacy parameter $\varepsilon$ growing as $\log(n)$, thus matching known lower bounds on the cost of privacy for this setting.
\end{abstract}

\section{Introduction}

Since its inception 20 years ago~\cite{Dwo06, dwork2006calibrating, dwork2008differential}, differential privacy has drawn a great deal of attention in both theory and practice. In addition to devising new algorithms with minimal privacy leakage~\cite{abadi2016deep, koskela2020computing, zhu2022optimal} and auditable guarantees~\cite{ding2018detecting, jagielski2020auditing, nasr2023tight, steinke2023privacy}, researchers have been interested in tightly characterizing tradeoffs between privacy, accuracy, and sometimes computational aspects of differentially private algorithms~\cite{bun2014fingerprinting, karwa2017finite, duchi2018minimax, kamath2019privately, cai2021cost, kamath2025bias}.

Despite the large volume of existing theoretical work on privacy, one area that has received relatively less attention is private estimation from graph inputs. Most previous work focuses on privacy with respect to edge perturbations~\cite{NisEtal07, hay2009accurate, karwa2011private, gupta2012iterative, blocki2012johnson}, while~\cite{karwa2012differentially, blocki2012johnson} developed general strategies for turning edge-private algorithms into node-private algorithms (i.e., two graphs are neighbors if one can be transformed to the other by altering connections to a single node).

In this paper, we study the problem of community estimation under a random graph model known as the stochastic block model (SBM)~\cite{holland1983stochastic}. The literature on (non-private) community estimation in stochastic block models is vast; we refer the reader to~\cite{abbe2018community} for an excellent survey. Notably, most existing work on private SBM estimation has focused on edge-private algorithms~\cite{HehEtal22, mohamed2022differentially, CheEtal23, chakraborty2024prime, nguyen2024differentially, koskela2025price}, while the focus in our paper will be on node-private algorithms which are minimax optimal in terms of statistical error.

\subsection{Related work}

Node-private estimation in SBMs was previously studied by~\cite{chen2024private}, who focused on estimating the $K \times K$ matrix of connection probabilities, rather than recovering communities. Our work is most directly related to \cite{klopp2026node}, who developed the first algorithm for private community estimation. However, the runtime of their  proposed algorithm is exponential in the size of the graph. Recently, \cite{marchis2026node} developed node-private community estimation algorithms which are computationally feasible; however, they focused on a weaker notion of consistent recovery. Their algorithmic approaches are completely separate from ours, as they are based on private spectral algorithms and smooth projections to the space of bounded-degree graphs, rather than direct sampling from a Gibbs distribution.

\subsection{Contributions}

Our main contribution is to analyze a novel algorithm for node-private community estimation in SBMs, which runs in time polynomial in the number of nodes $n$ and the number of communities $K$. The runtime is with high probability over the input adjacency matrix drawn from the SBM: We provide separate results with a polynomial expected runtime guarantee (Theorem~\ref{thm:runtime}), uniformly over the internal randomness of the algorithm, and a \emph{worst-case} runtime guarantee (Corollary~\ref{cor:choice_M_pure}), at the cost of a larger privacy parameter. Importantly, the output of our algorithm achieves the minimax risk of the exponential time algorithm in \cite{klopp2026node}. Our analysis allows for settings where $K$ grows logarithmically with $n$. By sharpening the utility analysis of \cite{klopp2026node}, we show that a privacy level of $\varepsilon = \Theta(\log(n))$ is both necessary and sufficient for the success of our algorithm.

Our private algorithm builds upon the exponential mechanism~\cite{mcsherry2007mechanism}, which samples a candidate with probability proportional to the exponential of a suitably scaled utility score.
%Thus, points with larger score are exponentially favored, while scaling by the score sensitivity ensures privacy. 
We choose the score so that near-maximizers retain utility guarantees which nearly match \cite{klopp2026node}. The main challenges are to show that: (1) we can define an efficiently computable score function which has low sensitivity on in-community bounded-degree graphs (implying privacy) and is sufficiently close to the penalized likelihood score (implying utility); and (2) we can sample from the corresponding exponential mechanism in a computationally tractable manner. We analyze these aspects in Sections~\ref{SecScore} and~\ref{SecSampling}, respectively, and derive the consequences in terms of high-probability runtime and utility in Section~\ref{SecSBMAnalysis}. Finally, we demonstrate how a truncated version of the rejection sampler can be used to obtain approximate samples from the desired Gibbs distribution with a polynomial runtime guarantee that is worst- rather than average-case (cf.\ Section~\ref{sec:truncated-sampler} and Appendix~\ref{AppApprox}). As we are aware of little work in the privacy literature which uses approximate sampling (cf.\ Remark~\ref{RemTruncated}), we believe this also constitutes an important conceptual contribution which may be of independent interest.

A key algorithmic ingredient is to decompose the Gibbs measure into pieces which are easy to sample from: we devise a proposal distribution expressed as a product measure over partitions, with an easy-to-implement rejection step, leveraging a carefully defined semidefinite program (SDP). We note that this technique is employed in many fields of sampling theory; see \cite{gilks1992adaptive} for a standard scheme for sampling one-dimensional Gibbs measures. In high dimensions, as in the present work, this approach is significantly more challenging. Another related example is how the Hubbard--Stratonovich decomposition \cite{hubbard1959calculation,stratonovich1958method} has been used to sample from complicated Gibbs measures of high-temperature Ising models, by decomposing the measure as a strongly log-concave mixture of product measures. This idea has been key for a series of striking recent algorithmic advances (see e.g., \cite{chen2022localization} and references therein). However, we are not aware of any similar decompositions used in differential privacy, and we hope the present paper will inspire further algorithmic work on efficient sampling from utility-optimal exponential mechanisms.

As our polynomial-time estimator achieves the minimax exact-recovery error with privacy budget matching the lower bound of
$\varepsilon=\Omega(\log(n))$~\cite{klopp2026node}, it is natural to wonder whether an alternative privacy notion could avoid
the need for a privacy parameter that grows with $n$. For example, \cite{bombari2025better} recently studied one-pass gradient
descent for private high-dimensional linear regression, and showed that it is possible to obtain non-trivial risk
guarantees with a constant zero-concentrated differential privacy (zCDP) budget~\cite{bombari2025better}. Their motivation to study zCDP is because standard gradient clipping would require a diverging privacy budget under the usual privacy paradigm.
In Appendix~\ref{AppZero}, we show that the SBM setting is fundamentally different: information-theoretically,
any estimator with polynomially small exact-recovery failure must have
privacy parameter $\rho=\Omega(\log(n))$, even under zCDP. Similar results are derived for approximate differential privacy in Appendix~\ref{AppLBApprox}. Thus, the need for a
diverging privacy budget is not a result of focusing solely on pure differential privacy. It is an interesting open question whether analogous lower bounds hold under other
notions of privacy.

%\begin{figure}[t]
%\centering
%\begin{tikzpicture}[
%    >=Latex,
%    every node/.style={font=\small},
%    box/.style={
%        draw,
%        rounded corners,
%        thick,
%        align=center,
%        minimum width=3.0cm,
%        minimum height=1.1cm,
%        inner sep=4pt
%    },
%    io/.style={
%        draw,
%        rounded corners,
%        thick,
%        align=center,
%        minimum width=2.4cm,
%        minimum height=1.0cm,
%        inner sep=4pt,
%        fill=gray!10
%    },
%    arrow/.style={->, thick}
%]
%
%\node[io] (input) {Input\\$A$};
%
%\node[box, below left=1.2cm and 0.9cm of input] (surrogate)
%{Likelihood surrogate\\$\tilde{T}_{A,D}$};
%
%\node[box, below right=1.2cm and 0.9cm of input] (candidate)
%{Candidate SDP\\labeling $\sigma^\star$};
%
%\node[box, below=1.5cm of input] (cert)
%{SDP certification\\Equation $\eqref{EqnCertify}$};
%
%\node[box, below=1.2cm of cert] (sampler)
%{Rejection sampler\\Algorithm \ref{alg:certified-sampler}};
%
%\node[io, below=1.2cm of sampler] (output)
%{Output\\$\hat{\sigma}$};
%
%\draw[arrow] (input.south west) -- (surrogate.north);
%\draw[arrow] (input.south east) -- (candidate.north);
%
%\draw[arrow] (surrogate.south east) -- (cert.west);
%\draw[arrow] (candidate.south west) -- (cert.east);
%
%\draw[arrow] (cert) -- (sampler);
%\draw[arrow] (sampler) -- (output);
%
%\end{tikzpicture}
%
%\caption{High-level structure of the proposed node-DP estimator (Algorithm \ref{alg:full})}
%\label{fig:algorithm_pipeline}
%\end{figure}

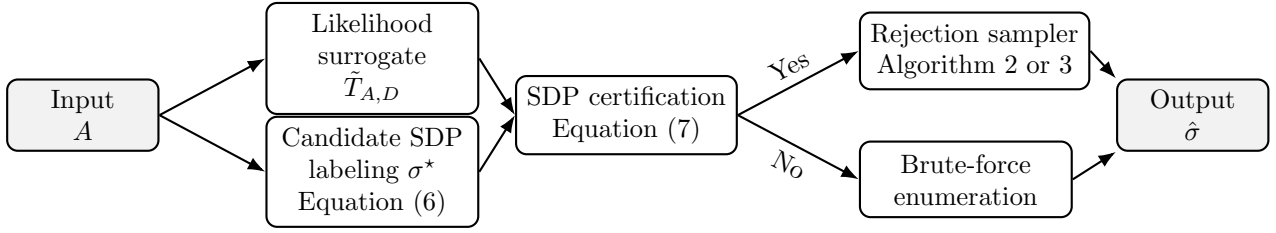
\begin{figure}[t]
\centering
\begin{tikzpicture}[
    >=Latex,
    every node/.style={font=\small},
    box/.style={
        draw,
        rounded corners,
        thick,
        align=center,
        minimum width=2.8cm,
        minimum height=1.0cm,
        inner sep=4pt
    },
    io/.style={
        draw,
        rounded corners,
        thick,
        align=center,
        minimum width=2.0cm,
        minimum height=0.9cm,
        inner sep=4pt,
        fill=gray!10
    },
    arrow/.style={->, thick}
]

\node[io] (input) {Input\\$A$};

\node[box, right=1.4cm of input, yshift=0.75cm] (surrogate)
{Likelihood \\ surrogate\\$\tilde{T}_{A,D}$};

\node[box, right=1.4cm of input, yshift=-0.75cm] (candidate)
{Candidate SDP\\labeling $\sigma^\star$ \\
Equation~\eqref{EqnCandSDP}};

\node[box, right=4.7cm of input] (cert)
{SDP certification\\Equation \eqref{EqnCertify}};

\node[box, right=1.6cm of cert, yshift=0.85cm] (sampler)
{Rejection sampler\\Algorithm \ref{alg:certified-sampler} or~\ref{alg:truncated-certified-sampler}};

\node[box, right=1.6cm of cert, yshift=-0.85cm] (bruteforce)
{Brute-force\\enumeration};

\node[io, right=5.0cm of cert] (output)
{Output\\$\hat{\sigma}$};

\draw[arrow] (input.east) -- (surrogate.west);
\draw[arrow] (input.east) -- (candidate.west);

\draw[arrow] (surrogate.east) -- (cert.west);
\draw[arrow] (candidate.east) -- (cert.west);

\draw[arrow] (cert.east) -- node[midway, above, sloped] {Yes} (sampler.west);
\draw[arrow] (cert.east) -- node[midway, below, sloped] {No} (bruteforce.west);

\draw[arrow] (sampler.east) -- (output.north west);
\draw[arrow] (bruteforce.east) -- (output.south west);

\end{tikzpicture}
\caption{High-level structure of the proposed node-private estimator (cf.\ Algorithm \ref{AlgOverview}).}
\label{FigAlg}
\end{figure}

%%%%%

\section{Background and problem setup}

We begin by defining the notion of differential privacy for undirected graph inputs which will be used in this paper. We also provide definitions, notation, and assumptions regarding SBMs. Additional notation is contained in Appendix~\ref{AppNotation}.

\subsection{Differential privacy}

We begin with the classical notion of $(\varepsilon, \delta)$-differential privacy, which will be denoted by $(\varepsilon, \delta)$-DP:
\begin{definition}
A randomized algorithm $\mathcal{A}$ satisfies $(\varepsilon, \delta)$-DP, for $\varepsilon, \delta \geq 0$, if for all pairs of datasets $X$ and $X'$ differing in one element and for all $S$ in the range of $\mathcal{A}$, we have $\mathbb{P}(\mathcal{A}(X) \in S) \leq e^{\varepsilon}\mathbb{P}(\mathcal{A}(X') \in S) + \delta$. If $\mathcal{A}$ satisfies $(\varepsilon, 0)$-DP, we also say it satisfies pure $\varepsilon$-DP.
\end{definition}  

Let \(\G_n\) be the set of undirected graphs on the vertex set \([n]\), represented by symmetric, zero-diagonal adjacency matrices \(A\in\{0,1\}^{n\times n}\).
Two graphs \(A,A'\in\G_n\) are \emph{node-adjacent}, written \(A\sim_v A'\), if
they differ only on edges incident to a single vertex. Let \(d_v(A,A')\) be the
shortest-path metric generated by this adjacency relation; equivalently,
\(d_v(A,A')\) is the minimum number of vertex neighborhoods that must be
rewired to transform \(A\) into \(A'\).

\begin{definition}
We say a randomized algorithm $\mathcal{A}$, evaluated on graph inputs in $\mathcal{G}_n$,  is $(\varepsilon, \delta)$-node DP if for any \(A\sim_v A'\) and measurable set $S$, we have
$\mathbb{P}(\mathcal{A}(A) \in S) \leq e^{\varepsilon}\mathbb{P}(\mathcal{A}(A') \in S) + \delta$.
\end{definition}

\subsection{Stochastic block models}

We now formally define the stochastic block model. Let $\sigma_0: [n] \rightarrow [K]$ denote the (unobserved) community assignments, where the nodes are partitioned into equal-sized communities of size $s = \frac{n}{K}$, so $\sigma_0 \in \Sigma
:=
\left\{
\sigma:[n]\to[K]:
\abs{\sigma^{-1}(k)}=s,
\; \forall k\in[K]
\right\}$.
%In other words, the $n$ nodes are divided into $K$ equally-sized communities, with within-community probability $\frac{a}{n}$ and between-community probability $\frac{b}{n}$. Let $s = \frac{n}{K}$. We will write $\sigma_0$ to denote the true community assignments:
Conditioned on $\sigma_0$, the upper-triangular entries of the adjacency matrix $A$ are generated independently according to
\[
\Pbb(A_{ij}=1\mid \sigma_0)=
\begin{cases}
 a/n, & \sigma_0(i)=\sigma_0(j),\\
 b/n, & \sigma_0(i)\ne \sigma_0(j),
\end{cases}
\]
for parameters $a,b \in \real$. This is the family of SBMs $\Theta(n,K,a,b,1)$ in the notation of~\cite{klopp2026node}.
For each \(\sigma\in\Sigma\), define the cluster matrix $Y^\sigma_{ij}
=
\1\{\sigma_i=\sigma_j\}$, for all $i,j\in[n]$.
Note that $Y^\sigma$ is invariant under relabeling of the communities.

%Define the set of exact-balanced labelings
%Fix integers \(n,K\) with \(K\mid n\), and set $s=\frac{n}{K}$. Let \(A\) be the adjacency matrix of an undirected graph on vertex set
%\([n]=\{1,\dots,n\}\). We write \(ij\in E(A)\) for an unordered edge \(\{i,j\}\) of \(A\). We work with exactly equal-size labelings:

%Since \(s\ge1\), every labeling in \(\Sigma\) uses every label.

We will be interested in measuring the accuracy of a community estimation algorithm with respect to the best community labeling. For $\sigma,\tau:[n]\to[K]$, we define the Hamming distance
$h(\sigma,\tau):=|\{i\in[n]:\sigma_i\ne\tau_i\}|$ and the permutation-invariant loss function $r(\sigma,\tau) :=\frac{\dorb([\sigma], [\tau])}{n}$, where $\dorb([\sigma],[\tau]) := \min_{\pi\in\Sn_K} h(\pi\circ\sigma,\tau)$.
%Every exact-size cluster matrix belongs to \(\cC\). Indeed, if
%\(\sigma\in\Sigma\), then \(Y^\sigma\) is %positive semidefinite because it is
%a Gram matrix of membership vectors, it has diagonal \(1\), nonnegative
%entries, and
%\[
%(Y^\sigma\one)_i
%=
%\abs{\sigma^{-1}(\sigma_i)}
%=
%s.
%\]

%A useful simplification in the exact-size case is that
%\[
%\sum_{i<j}Y_{ij}
%\]
%is constant over \(\cC\). Indeed,

We state the following assumptions on the parameters of the SBM, which are required for our analysis of the utility and runtime of our algorithms (but not the privacy guarantees):
\begin{assumption}
\label{ass:mild-K}
There exists an absolute constant $C_{\mathrm{mg}} > 0$ such that $K\log(K)\le C_{\mathrm{mg}}\log(n)$.
\end{assumption}

\begin{assumption}
\label{ass:high-signal}
There exist absolute constants $0 < \rho_- \le \rho_+ < 1$ and $A_0 > 0$ such that $\frac aK\ge A_0\log(nK)$, $a,b=o(n)$, and $ 0<\rho_-\le b/a\le \rho_+<1$.
\end{assumption}

If $I$ denotes the order-$1/2$ Renyi divergence between $\Ber(a/n)$ and $\Ber(b/n)$, then under Assumptions~\ref{ass:mild-K} and~\ref{ass:high-signal}, we have
%the non-private error exponent is of order 
%If $\frac{\mathrm{Signal}}{\log(K)} \rightarrow \infty$  (which is implied by Assumptions \ref{ass:mild-K} and \ref{ass:high-signal}), then 
$\inf_{\hat{\sigma}}\sup_{\sigma_0\in\Sigma}\Ee r(\sigma_0,\widehat\sigma) = e^{-(1 + o(1))\mathrm{Signal}}$,
where $\mathrm{Signal}:=\frac{nI}{K}\asymp \frac aK$ and the expectation is taken with respect to the randomness in $\hat{\sigma}$ and the SBM generated with $\sigma_0$~\cite{zhang2016minimax, klopp2026node}.
%Note that if $b < a = o(n)$ (again implied by Assumption \ref{ass:high-signal}), we have $\mathrm{Signal} = (1 + o(1))(\sqrt{a} - \sqrt{b})^2/K$. Furthermore, the $e^{-(1 + o(1))\mathrm{Signal}}$
The minimax rate can be attained by penalized likelihood methods~\cite{zhang2016minimax}.

%%%%%

\section{Score function}
\label{SecScore}

We introduce the critical choice of a score function for the exponential mechanism. One natural choice is the penalized likelihood score $T_A(\sigma) := \sum_{i<j}(A_{ij}-\lambda)\1\{\sigma(i)=\sigma(j)\}$, which achieves the non-private minimax rate~\cite{zhang2016minimax} and was also utilized by~\cite{klopp2026node} to obtain the private minimax rate. Note that for $\sigma \in \Sigma$, the penalty term $\lambda\sum_{i<j}\mathbf 1\{\sigma(i)=\sigma(j)\}=\lambda K\binom{n/K}{2}$ is constant; however, we retain it to remain consistent with~\cite{zhang2016minimax,klopp2026node}. For completeness, $\lambda$ is chosen as in \cite[Section 3.1]{klopp2026node}. 

Towards building an efficient algorithm, our first step is to introduce a linear programming (LP) proxy of $T_A(\sigma)$. For $D>0$ and $Y\in[0,1]^{n\times n}$, we define
\begin{equation}
\Hin_{A,D}(Y)
:=
\min_{q,z}
\left\{
\sum_{ij\in E(A)} z_{ij}Y_{ij}
+
D\sum_{i=1}^n q_i
\right\},
\label{eq:H-tilde}
\end{equation}
subject to
\begin{equation}
        0\le q_i\le 1\quad \forall i,
        \qquad
        0\le z_{ij}\le 1\quad \forall ij\in E(A),
        \qquad
        z_{ij}\ge 1-q_i-q_j\quad \forall ij\in E(A).
\label{eq:H-constraints}
\end{equation}
We also define $\Psiin_{A,D}(Y)
        :=
        \Hin_{A,D}(Y)-\lambda\sum_{i<j}Y_{ij}$, and finally, the LP proxy
$\Tin_{A,D}(\sigma):=\Psiin_{A,D}(Y^\sigma)$.

\begin{remark}
\label{lem:score-eval}
For fixed $A,D$, and $Y$, we can compute $\Hin_{A,D}(Y)$ using a linear program with $n+|E(A)|$ variables and $O(n+|E(A)|)$ constraints, so $\Tin_{A,D}(\sigma)$ can be evaluated in polynomial time.
\end{remark}

%\begin{proof}
%This is immediate from \eqref{eq:H-tilde}--\eqref{eq:H-constraints}.  The variables are $q_i$, $i\in[n]$, and $z_{ij}$, $ij\in E(A)$, and all constraints and the objective are linear.
%\end{proof}

Importantly, for ``typical" graph inputs $A$, we will show that the score function $\Tin_{A,D}$ is close to $T_A$ for near-maximizers.
Define the in-community bounded degree event, for a fixed $D > 0$, as
\begin{equation}
\cE_D
:=
\left\{
\max_{i\in[n]}
\sum_{j\ne i}A_{ij}\ind\{\sigma_0(j)=\sigma_0(i)\}
\le D
\right\},
\label{eq:true-within-degree-event}
\end{equation}
which we will later show (Lemma~\ref{lem:true-within-conc}) occurs w.h.p.\ when $A$ is sampled from an SBM. The proof of the following result is contained in Appendix~\ref{AppLemTruth}:

\begin{lemma}[Near-optimizer proximity]
\label{lem:truth-agreement}
For every graph $A$ and every $\sigma\in\Sigma$, we have
\begin{equation}
        \Tin_{A,D}(\sigma)\le T_A(\sigma).
\label{eq:one-sided}
\end{equation}
Furthermore, on $\cE_D$, we have $\Tin_{A,D}(\sigma_0) = T_A(\sigma_0)$; in particular, if $\Tin_{A,D}(\sigma)\ge\max_{\tau\in\Sigma}\Tin_{A,D}(\tau)-u$, we have $T_A(\sigma)\ge T_A(\sigma_0)-u$.
\end{lemma}

Note that the conclusion of Lemma~\ref{lem:truth-agreement} is not that $\Tin_{A,D}$ is simply an ``extension" of $T_A$; the lemma merely states that $\Tin_{A,D}(\sigma_0) = T_A(\sigma_0)$ and guarantees closeness of the two functions at near-optimizers. For a more complete discussion of the connection to Lipschitz extensions, see Remark~\ref{RemLipExt} below.
Nonetheless, we show that $\Tin_{A,D}$ has bounded sensitivity, leading to a private Gibbs sampling mechanism. The proof of the following result is contained in Appendix~\ref{AppLemGlobal}:

\begin{lemma}
\label{lem:global-sensitivity}
For every pair of graphs $(A,A')$, we have $|\Tin_{A,D}(\sigma)-\Tin_{A',D}(\sigma)|
        \le
        D\,d_v(A,A')$, for all $\sigma\in\Sigma$.
Thus, the mechanism which samples from the distribution
\begin{equation}
        \pi_A(\sigma)
        :=
        \frac{\exp\{\eta\Tin_{A,D}(\sigma)\}}
        {\sum_{\tau\in\Sigma}\exp\{\eta\Tin_{A,D}(\tau)\}},
        \qquad
        \eta:=\frac{\varepsilon}{2D}.
\label{eq:gibbs-law}
\end{equation}
is pure $\varepsilon$-node-DP.
\end{lemma}

The output of our private algorithm is a sample from $\pi_A$. In what follows, we describe a polynomial-time sampler and analyze its runtime and utility. Our key algorithmic idea is to decompose $\pi_A$ as a mixture of a product measure proposal step and an SDP-based acceptance filter.

\section{Efficient sampling}
\label{SecSampling}

Whereas sampling directly from the Gibbs distribution~\eqref{eq:gibbs-law} may purportedly take exponential time, we show that it is possible to construct a high-probability polynomial-time algorithm based on sequential rejection sampling: First, we use an SDP to compute a candidate labeling $\sigma^\star$, which with high probability (cf.\ Lemma~\ref{lem:eroded-cert-highprob}) is equal to $\sigma_0$. We then run another SDP to certify that $\sigma^\star$ maximizes the penalized likelihood (cf.\ Lemma~\ref{lem:eroded-sdp}), after which we can sample from $\pi_A$ using an acceptance filter. Details are provided in Algorithm~\ref{AlgOverview} (see also Figure~\ref{FigAlg}).

%\textcolor{red}{IZ: OK so there is some $\sigma^*$ input to this method. The importance of this choice is unclear to the reader until much later, no? Maybe I miss something, but why dont we start with the candidate SDP rightaway. }

\begin{comment}
\begin{enumerate}
\item Let \(\sigma^\star\in\Sigma\) be a candidate labeling and set $Y^\star=Y^{\sigma^\star}$.
%Define $L_\star(Y)=\norm{Y-Y^\star}_1$.
%Since \(Y^\star_{ij}\in\{0,1\}\) and \(0\le Y_{ij}\le1\) on \(\cC\), we have
%\[
%L_\star(Y)
%=
%\sum_{Y^\star_{ij}=1}(1-Y_{ij})
%+
%\sum_{Y^\star_{ij}=0}Y_{ij}.
%\]
%Thus, \(L_\star\) is affine on \(\cC\).
For a cluster matrix $Z$ and a parameter $\theta > 0$, define
\begin{equation*}
        V_\theta(Z)
        :=
        \max_{Y\in\cC}
        \left\{
        \Psiin_{A,D}(Y)+\frac{\theta}{n}\|Y-Z\|_1
        \right\},
\end{equation*}
where $\cC
: =
\left\{
Y\succeq0,\;
Y_{ii}=1\ \forall i,\;
Y_{ij}\ge0\ \forall i,j,\;
Y\1=s\1
\right\}$.
%Together with \(Y_{ij}\ge0\), this gives
%\[
%0\le Y_{ij}\le1
%\qquad\forall i,j.
%\]
%It is easy to check that $Y^\sigma \in \mathcal{C}$ for all $\sigma \in \Sigma$. 
\item Check if
\begin{equation}
\label{EqnCertify}
V_\theta(Y^\star) = \Psiin_{A,D}(Y^\star):
\end{equation}
\begin{itemize}
\item If equality holds, run the certified rejection sampler (Algorithm~\ref{alg:certified-sampler}) with parameter $\kappa \propto \theta$.
\item Otherwise, sample exactly from \(\pi_A\) by brute-force enumeration of
\(\Sigma\).
\end{itemize}
\end{enumerate}
\end{comment}

\begin{algorithm}[h]
\caption{Main algorithm}
\label{AlgOverview}
\begin{algorithmic}[1]
\STATE Input graph $A$, parameters $K,D,\theta,\varepsilon$, $\eta=\frac{\varepsilon}{2D}$, and $\kappa = 2\eta\theta/K$ so that $\kappa\ge \log(4n(K-1))$.
\STATE Solve the candidate SDP
\begin{equation}
\label{EqnCandSDP}
\max_X\langle A,X\rangle
\quad\text{s.t.}\quad
X\succeq0,
\quad X\one\le\one,
\quad \Tr(X)=K,
\quad X\ge0.
\end{equation}
\STATE If $s\widehat X$ is an exact-size cluster matrix, read off some $\sigma^\star\in\Sigma$ such that $Y^{\sigma^\star}=s\widehat X$.  Otherwise, choose an arbitrary $\sigma^\star\in\Sigma$. Set $Y^\star=Y^{\sigma^\star}$.
\STATE For a cluster matrix $Z$, define
$V_\theta(Z)
        :=
        \max_{Y\in\cC}
        \left\{
        \Psiin_{A,D}(Y)+\frac{\theta}{n}\|Y-Z\|_1
        \right\}$,
where $\cC
: =
\left\{
Y\succeq0,\;
Y_{ii}=1\ \forall i,\;
Y_{ij}\ge0\ \forall i,j,\;
Y\1=s\1
\right\}$.
\STATE Check if
\begin{equation}
\label{EqnCertify}
V_\theta(Y^\star) = \Psiin_{A,D}(Y^\star).
\end{equation}
\STATE If condition \eqref{EqnCertify} holds, run the certified rejection sampler (Algorithm~\ref{alg:certified-sampler}) with parameter $\kappa$.
\STATE Otherwise, sample exactly from \(\pi_A\) by brute-force enumeration of
\(\Sigma\).
\end{algorithmic}
\end{algorithm}

Theorem~\ref{thm:runtime} below shows that Algorithm~\ref{alg:certified-sampler} runs in polynomial time when the certification condition~\eqref{EqnCertify} is satisfied. The algorithm also involves checking if $\sigma$ is contained in a set $R_{\sigma^\star}$ of canonical permutations, the definition of which will be provided in Section~\ref{SecHungarian}. Only if the certification condition does not hold (which we show in Lemma~\ref{lem:eroded-cert-highprob} to be a low-probability event), do we sample exactly from $\pi_A$ by brute-force enumeration, which takes exponential time.

%defined as follows: For every orbit \([\sigma]\) under label permutations, choose the representative
%\(\tau\in[\sigma]\) minimizing $h(\tau,\sigma^\star)$,
%with deterministic tie-breaking. Let \%(\cR_{\sigma^\star}\) be the set of such
%canonical representatives. In particular, for \(\sigma \in \cR_{\sigma^\star}\), we have
%\[
%h(\sigma,\sigma^\star)
%=
%d([\sigma],[\sigma^\star]).
%\]

\begin{algorithm}[h]
\caption{Certified rejection sampler}
\label{alg:certified-sampler}
\begin{algorithmic}[1]
\STATE For each vertex $i$, independently propose a label $\sigma_i$ by
\[
\Pp(\sigma_i=\sigma_i^\star)=\frac{1}{1+(K-1)e^{-\kappa}},
\qquad
\Pp(\sigma_i=\ell)=\frac{e^{-\kappa}}{1+(K-1)e^{-\kappa}}
\quad(\ell\ne\sigma_i^\star).
\]
\STATE If $\sigma\notin\Sigma$ or $\sigma\notin R_{\sigma^\star}$, reject and restart.
\STATE Accept $\sigma$ with probability $b_A(\sigma)
        :=
        \exp\left\{
        \eta[\Tin_{A,D}(\sigma)-\Tin_{A,D}(\sigma^\star)]
        +\kappa h(\sigma,\sigma^\star)
        \right\}$.
If rejected, restart.
\STATE If $\sigma$ is accepted, draw $\Pi\sim \operatorname{Unif}(\mathfrak S_K)$ and return $\Pi\circ\sigma$.
\end{algorithmic}
\end{algorithm}

\begin{remark}
\label{remark:sampling_poly_unif_perm}
Algorithm~\ref{alg:certified-sampler} also involves sampling a uniform permutation of $K$ labels from the set $\Sn_K$. Although $|\Sn_K| = K!$, selecting a random permutation can be done in $O(K^2)$ time, assuming it is possible to generate Unif$[0,1]$ random variables in constant time.
\end{remark}

Before proceeding, we prove that the certification step can be executed in polynomial time. The proof of the following result is contained in Appendix~\ref{AppLemEroded}:

\begin{lemma}
\label{lem:eroded-sdp}
If $Z$ is a cluster matrix, then $V_\theta(Z)$ is the value of a polynomial-size SDP.
\end{lemma}

In the subsections that follow, we analyze each step in turn. The main result is that, for an appropriate choice of $\theta$, success of the certification condition ensures that the output of Algorithm~\ref{alg:certified-sampler} samples from the target distribution, and can be executed in (expected) polynomial time.

\begin{remark}
For numerical implementation, the exact equality of SDP optima in condition~\eqref{EqnCertify} can be delicate. One could instead allow a slack $\rho>0$ and require $V_\theta(Y^\star) \le \Psiin_{A,D}(Y^\star) + \rho$. The analysis with such a slack is analogous to what follows. In fact, with high probability, condition~\eqref{EqnCertify} can be satisfied exactly. Since our focus is improving computational complexity rather than the mechanics of numerical optimization, we use the exact condition~\eqref{EqnCertify}.
\end{remark}

%%%%%

\subsection{Canonical permutations}
\label{SecHungarian}

For $b_A(\sigma)$ to be a valid acceptance probability, we need
$\eta[\Tin_{A,D}(\sigma)-\Tin_{A,D}(\sigma^\star)]
+\kappa h(\sigma,\sigma^\star)\le 0$.
However, this inequality cannot hold over all of \(\Sigma\), since permuting the
community labels of \(\sigma^\star\) leaves \(\Tin_{A,D}\) unchanged, but may
produce a labeling with a large Hamming distance from \(\sigma^\star\). To remedy this issue, Algorithm~\ref{alg:certified-sampler} restricts attention to a set of permutations $R_{\sigma^\star}$, which we now define.
%In Lemma~\ref{lem:certificate-margin} below, we prove that $\Tin_{A,D}$ has the required curvature on permutations in $R_{\sigma^*}$.
For \(\sigma\in\Sigma\), define the set of canonical representatives
\[
        \operatorname{Can}_{\sigma^\star}(\sigma)
        :=
        \text{the lexicographically smallest element of }
        \arg\min_{\pi\in\mathfrak S_K}
        h(\pi\circ\sigma,\sigma^\star),
\]
where lexicographic order $<_{\rm lex}$ means that
for \(\tau,\tau'\in[K]^n\), we have \(\tau<_{\rm lex}\tau'\) if, at the first
coordinate \(i\) for which \(\tau_i\ne\tau'_i\), one has \(\tau_i<\tau'_i\).
Define $R_{\sigma^\star} :=
        \{\sigma\in\Sigma:
        \operatorname{Can}_{\sigma^\star}(\sigma)=\sigma\}$.

The following result, proved in Appendix~\ref{AppLemRomanian}, shows that it is possible to check if a labeling is contained in $R_{\sigma^\star}$ in polynomial-time. It involves an application of the Hungarian algorithm~\cite{papadimitriou1998combinatorial}. 

\begin{lemma}
\label{lem:hungarian-canon}
If \(\sigma\in R_{\sigma^\star}\), we have $h(\sigma,\sigma^\star) = \dorb([\sigma],[\sigma^\star])$.
For any $\sigma \in \Sigma$, the representative \(\operatorname{Can}_{\sigma^\star}(\sigma)\), and hence membership in
\(R_{\sigma^\star}\), can be computed in polynomial time.
\end{lemma}

%%%%%

\subsection{Certified rejection sampler}

We now turn to the analysis of Algorithm~\ref{alg:certified-sampler}.
Assume a candidate $\sigma^\star\in\Sigma$ satisfies the certificate \eqref{EqnCertify} with $Y^\star = Y^{\sigma^\star}$.
%Define $\Theta:=\frac{2\theta}{K}$ and $\kappa:=\eta\Theta$.
We begin with a technical lemma, proved in Appendix~\ref{AppLemGeom}, concerning the geometry of the space of cluster matrices:

\begin{lemma}
\label{lem:geometry}
For every $\sigma,\tau\in\Sigma$, we have $\|Y^\sigma-Y^\tau\|_1
        \ge
        \frac{2n}{K}\dorb([\sigma],[\tau])$.
\end{lemma}

Using Lemma~\ref{lem:geometry}, we show that success of the certification condition implies a score margin, leading to a valid acceptance probability. The following result is proved in Appendix~\ref{AppCertMar}:

\begin{lemma}[SDP certification implies stable LP likelihood maximization]
\label{lem:certificate-margin}
Suppose $\sigma^\star\in\Sigma$ satisfies
$V_\theta(Y^{\sigma^\star})=\Psiin_{A,D}(Y^{\sigma^\star})$.
Then for every $\sigma\in\Sigma$, we have
\begin{equation}
\label{eq:certified-margin-theta}
        \Tin_{A,D}(\sigma^\star)-\Tin_{A,D}(\sigma)
        \ge
        \frac{2\theta}{K}\,\dorb([\sigma],[\sigma^\star]).
\end{equation}
In particular, if $\kappa = \frac{2\eta\theta}{K}$, we have $b_A(\sigma) \le 1$ for every $\sigma \in \Sigma \cap R_{\sigma^\star}$.
\end{lemma}

%By Lemma \ref{lem:certificate-margin},
%\begin{equation}
%        \Tin_{A,D}(\sigma^\star)-\Tin_{A,D}(\sigma)
%        \ge
%        \Theta\dorb([\sigma],[\sigma^\star])
%        \qquad\forall\sigma\in\Sigma.
%\label{eq:certified-margin-theta}
%\end{equation}

%By Lemma \ref{lem:certificate-margin}, we have
%\begin{equation}
%        \Tin_{A,D}(\sigma^\star)-\Tin_{A,D}(\sigma)
%        \ge
%        \Theta\dorb([\sigma],[\sigma^\star])
%        \qquad\forall\sigma\in\Sigma.
%\label{eq:certified-margin-theta}
%\end{equation}

%We now show that the acceptance probability is indeed at most 1:

%\begin{lemma}
%\label{lem:accept-valid}
%If equation~\eqref{eq:certified-margin-theta} holds and $\kappa=\eta\Theta$, then $b_A(\sigma)\le1$ for every proposed $\sigma\in\Sigma\cap R_{\sigma^\star}$.
%\end{lemma}

%\begin{proof}
%For $\sigma\in R_{\sigma^\star}$, Lemma \ref{lem:hungarian-canon} gives $h(\sigma,\sigma^\star)=\dorb([\sigma],[\sigma^\star])$. Thus, inequality~\eqref{eq:certified-margin-theta} gives
%\[
%        \Tin_{A,D}(\sigma)-\Tin_{A,D}(\sigma^\star)
%        \le
%        -\Theta h(\sigma,\sigma^\star).
%\]
%Multiplying by $\eta$ and using $\kappa=\eta\Theta$ yields
%\[
%        \eta[\Tin_{A,D}(\sigma)-\Tin_{A,D}(\sigma^\star)]
%        +\kappa h(\sigma,\sigma^\star)
%        \le 0.
%\]
%Hence, $b_A(\sigma)\le1$.
%\end{proof}

Next, we show that the output of Algorithm~\ref{alg:certified-sampler} matches the desired Gibbs distribution, with a constant bound on the expected number of steps. The following result is proved in Appendix~\ref{AppLemExact}:

\begin{lemma}
\label{lem:sampler-exact}
Under the hypotheses of Lemma \ref{lem:certificate-margin}, Algorithm \ref{alg:certified-sampler} outputs an exact sample from the Gibbs law \eqref{eq:gibbs-law}, and, if $\kappa \ge \log(4n(K-1))$, the expected number of proposal trials is at most $e^{1/4}$.
\end{lemma}

\section{Analysis of SBMs}
\label{SecSBMAnalysis}

In this section, we will show that with high probability over the distribution of adjacency matrices from an SBM, it is possible to produce a candidate labeling $\sigma^\star$ for which the certification step succeeds. The analysis is based on a result from~\cite{pirinen2019exact} for exact recovery for SBMs via the SDP~\eqref{EqnCandSDP}. This leads to an overall guarantee that, with high probability over the inputs, Algorithm~\ref{AlgOverview} has polynomial expected runtime. We also show that the output of Algorithm~\ref{AlgOverview} retains the minimax optimal private misclassification rate achieved by the exponential-time mechanism in~\cite{klopp2026node}.

%We first state the algorithm (Algorithm~\ref{alg:full}), followed by the analysis of its runtime and utility.

%\begin{algorithm}[h]
%\caption{Node-private sampler}
%\label{alg:full}
%\begin{algorithmic}[1]
%\STATE Input graph $A$, parameters $K,D,\theta,\varepsilon$, and $\eta=\frac{\varepsilon}{2D}$.
%\STATE Solve the candidate SDP
%\[
%\max_X\langle A,X\rangle
%\quad\text{s.t.}\quad
%X\succeq0,
%\quad X\one\le\one,
%\quad \Tr(X)=K,
%\quad X\ge0.
%\]
%\STATE If $s\widehat X$ is an exact-size cluster matrix, read off some $\sigma^\star\in\Sigma$ such that $Y^{\sigma^\star}=s\widehat X$.  Otherwise, choose an arbitrary $\sigma^\star\in\Sigma$.
%\STATE Compute $V_\theta(Y^{\sigma^\star})$ using the SDP from Lemma \ref{lem:eroded-sdp}.
%\STATE If $V_\theta(Y^{\sigma^\star})=\Psiin_{A,D}(Y^{\sigma^\star})$,
%and $\kappa=\frac{2\eta\theta}{K}$ satisfies $\kappa\ge \log(4n(K-1))$,
%run Algorithm \ref{alg:certified-sampler} with $\kappa=\frac{2\eta\theta}{K}$.
%\STATE Otherwise, sample exactly from $\pi_A$ by brute-force enumeration of $\Sigma$.
%\end{algorithmic}
%\end{algorithm}

%%%%%

%\subsection{Preliminary results}
\subsection{High-probability expected runtime}

The following result, proved in Appendix~\ref{AppLemTrue}, lower-bounds the probability of $\cE_D$ (cf.\ equation~\eqref{eq:true-within-degree-event}):

\begin{lemma}
\label{lem:true-within-conc}
There are constants $C_{deg},c_{deg}>0$ such that if $D=C_{deg}\frac aK$ and Assumption \ref{ass:high-signal} holds with $A_0$ sufficiently large, we have $\Pp_{\sigma_0}(\cE_D^c)\le(nK)^{-c_{deg} A_0}$.
\end{lemma}

The following result, proved in Appendix~\ref{AppErodedCert}, shows that the certification condition~\eqref{EqnCertify} is satisfied with high probability:

\begin{lemma}
\label{lem:eroded-cert-highprob}
Let $D=C_{deg}\frac aK$ and $\theta=\chi a$, where $C_{deg}$ is sufficiently large and $\chi>0$ is chosen small enough that $a-b-4\theta\ge c_\rho a$, for some constant $c_\rho>0$ depending only on $\rho_+$. Under Assumptions \ref{ass:mild-K} and \ref{ass:high-signal}, with $A_0$ sufficiently large, the true cluster matrix $Y^{\sigma_0}$ satisfies the certification condition~\eqref{EqnCertify}, with probability at least $1-n^{-c} - (nK)^{-cA_0}$.
%\begin{equation}
%        V_\theta(Y^0)=\Psiin_{A,D}(Y^0).
%\label{eq:true-cert}
%\end{equation}
\end{lemma}

The main result of this section,
proved in Appendix~\ref{AppThmRuntime}, summarizes the privacy and high-probability runtime guarantees of Algorithm~\ref{AlgOverview} for SBM inputs:

\begin{theorem}[Privacy and runtime guarantees]
\label{thm:runtime}
Consider Assumptions \ref{ass:mild-K} and \ref{ass:high-signal}. Choose $D=C_{deg}\frac aK$ and $\theta=\chi a$,
where $C_{deg}$ is sufficiently large and $\chi>0$ is sufficiently small depending only on $\rho_+$.  There exist constants $L,c>0$ such that if $\varepsilon\ge L\log(nK)$, Algorithm~\ref{AlgOverview} has the following properties:
\begin{enumerate}[label=(\roman*)]
\item For every input graph, it outputs an exact sample from $\pi_A$ in \eqref{eq:gibbs-law}.
\item For every input graph, it is pure $\varepsilon$-node-DP.
\item With probability at least $1-n^{-c}-(nK)^{-cA_0}$ over the SBM input, the certified branch is used and the conditional expected runtime, over the internal randomness of the sampler, is $\poly(n)$.
\end{enumerate}
\end{theorem}

\subsection{Utility analysis}
\label{SecSBM}

We now turn to analyzing the utility of the output of Algorithm~\ref{AlgOverview}. The proof of the following result is provided in Appendix~\ref{AppUtilityProof}. It requires carefully adapting the proofs of Lemmas A.1 and A.5 in \cite{klopp2026node} (cf.\ Lemmas \ref{lemma:raw-kz} and \ref{lem:peeling}), because our
exponential mechanism is run with the modified score $\widetilde T_{A,D}$, not with the raw score $T_A$. On the event $\mathcal E_D$, by
Lemma~\ref{lem:truth-agreement}, if $\sigma$ is within a slack $u$ of maximizing $\widetilde T_{A,D}$, i.e., $\Tin_{A,D}(\sigma)\ge\max_{\tau\in\Sigma}\Tin_{A,D}(\tau)-u$, then $T_A(\sigma)\ge T_A(\sigma_0)-u$. However, this does not imply that $T_A(\sigma) \ge \max_{\tau\in\Sigma}T_A(\tau) - u$, which is required by \cite{klopp2026node}. Hence, the utility analysis requires estimates for the set $\{\sigma\in\Sigma:T_A(\sigma)\ge T_A(\sigma_0)-u\}$, rather than only for the near-maximum set $\{\sigma\in\Sigma:T_A(\sigma)\ge \max_{\tau\in\Sigma}T_A(\tau)-u\}$.

\begin{theorem}[Risk guarantees]
\label{thm:risk}
Suppose Assumptions \ref{ass:mild-K} and \ref{ass:high-signal} hold and $K\ge2$. Let $D=C_{deg}\frac aK$ and $\theta=\chi a$, where $C_{deg}$ is sufficiently large and $\chi>0$ is sufficiently small, depending on $\rho_+$. There exist constants $L,c_1,c_2,c_3>0$ such that if $\varepsilon\ge L\log(nK)$, the output $\widehat\sigma$ of Algorithm~\ref{AlgOverview} satisfies
\begin{equation}
        \sup_{\sigma_0\in\Sigma}\Ee r(\sigma_0,\widehat\sigma)
        \le
        \exp\left\{-c_1\frac{nI}{K}\right\}
        +
        \frac1{nK}e^{-c_2\varepsilon}
        +
        (nK)^{-c_3A_0}.
\label{eq:risk-final}
\end{equation}
\end{theorem}

Combined with the lower bound in \cite[Theorem 4.1]{klopp2026node}, Theorem \ref{thm:risk} shows that, for $a \asymp K\log(n)$, for any fixed $c>0$ and $A_0$ sufficiently large, a privacy level of $\varepsilon=\Theta(\log(n))$ is necessary and sufficient to drive the expected misclassification minimax risk $\mathop{\inf}\limits_{\widehat\sigma:\varepsilon\text{-node-DP}}
\mathop{\sup}\limits_{\sigma_0\in\Sigma}
\mathbb E r(\sigma_0,\widehat\sigma)$ below $n^{-(1+c)}$. Theorem~\ref{thm:risk} significantly sharpens the results of \cite{klopp2026node}, in which the proposed private mechanism is exponential-time and requires $\varepsilon \succsim K \log(n)$.

\begin{remark}
\label{RemLipExt}
The score $\widetilde T_{A,D}$ should not be viewed as a Lipschitz extension of $T_A$, previously introduced in the literature, in the usual sense of agreeing with $T_A$ on all candidate labelings on a high-probability domain. The reason is that, in the growing-$K$ regime, the relevant sensitivity scale is $\frac{a}{K}$, while
the total degree scale is $a$. A uniform agreement statement $\widetilde T_{A,D}(\sigma)=T_A(\sigma)$, for all $\sigma \in \Sigma$,
would require controlling the
number of incident edges from a vertex to vertices placed in the same
$\sigma$-community. Uniformly over all $\sigma$, this is controlled by
the total degree and is therefore of order $a$, not $\frac{a}{K}$. Using this larger scale would introduce an additional factor of $K$ in the lower bound on $\varepsilon$, resulting in $\varepsilon = \Omega(K\log(nK))$, since we would need to take $D\asymp a$, so $\eta=\frac{\varepsilon}{2D}\asymp \frac{\varepsilon}{a}$. Since the proof of Theorem \ref{thm:risk} requires $\eta\gtrsim K\log(nK)/a$, this would in turn require $\varepsilon\gtrsim K\log(nK)$.

Instead, on the event
$\mathcal \cE_D$, Lemma~\ref{lem:truth-agreement} gives $\widetilde T_{A,D}(\sigma)\le T_A(\sigma)$, for all $\sigma\in\Sigma$, with $\widetilde T_{A,D}(\sigma_0)=T_A(\sigma_0)$.
This is sufficient for the utility analysis: If $\widetilde T_{A,D}(\sigma)
        \ge
        \max_{\tau\in\Sigma}\widetilde T_{A,D}(\tau)-u$,
we have
$$\widetilde T_{A,D}(\sigma)
        \ge
        \widetilde T_{A,D}(\sigma_0)-u
        =
        T_A(\sigma_0)-u.$$
Since $T_A(\sigma)\ge \widetilde T_{A,D}(\sigma)$, it follows that $T_A(\sigma)\ge T_A(\sigma_0)-u$. Thus, every near-maximizer of the modified score is a near-maximizer of the raw score. The peeling argument in Lemma \ref{lem:peeling} shows that the Gibbs sampler outputs a near-maximizer of $\widetilde T_{A,D}$ with high probability,
and Lemma \ref{lem:true-relative-raw-slack} converts this into the desired risk bound.
\end{remark}

\begin{comment}
\subsection{Why the privacy condition is $\varepsilon\gtrsim\log(nK)$}

There are two analytic conditions involving $\varepsilon$. First, the peeling condition is $\eta>2B$.  With $D\asymp a/K$,
\[
        \eta=\frac{\varepsilon}{2D}\asymp \frac{\varepsilon K}{a},
        \qquad
        B\asymp\frac{K\log(nK)}{a},
\]
so $\eta>2B$ is equivalent to $\varepsilon\gtrsim \log(nK)$. Second, the sampler low-temperature condition is
\[
        \kappa=\eta\frac{2\theta}{K}\ge\log(4n(K-1)).
\]
With $\theta\asymp a$ and $D\asymp a/K$, we have $\kappa \asymp \frac{\varepsilon K}{a}\cdot \frac{a}{K} \asymp \varepsilon$. 
\end{comment}

%%%%%

\section{Truncated rejection sampler}
\label{sec:truncated-sampler}

As noted in Theorem~\ref{thm:runtime}, the expected runtime of Algorithm~\ref{AlgOverview}, assuming the certification condition is met (which occurs w.h.p.\ over SBM inputs) is polynomial in $n$. In fact, we can derive a stronger algorithmic guarantee, turning the runtime under the certification condition into \emph{worst-case} polynomial time uniformly over the internal randomness of the algorithm. The analysis of the truncated sampler (Algorithm~\ref{alg:truncated-certified-sampler}) is inspired by~\cite[Lemma 1]{cao2024channel}, which derives a total variation (TV) distance approximation guarantee on the output of the truncated sampler. In order to preserve \emph{pure} differential privacy, we require a more careful construction of the fallback sampling distribution (in case no proposal is accepted) and the behavior of the output distribution. For completeness, in Appendix~\ref{AppTruncTV}, we include the analysis of a truncated sampler based on TV distance, resulting in an approximate-DP guarantee.

\begin{remark}
\label{RemTruncated}
\cite{wang2015privacy, minami2016differential} show that a mechanism whose output law is close, uniformly over input datasets, in TV distance to that
of a DP mechanism, satisfies approximate DP. Our analysis in Appendix~\ref{AppTruncTV} produces similar $(\varepsilon, \delta)$ terms, where $\delta = (1 + e^{\varepsilon})\zeta_0$, and $\zeta_0$ is a bound on the TV distance. Truncated rejection samplers for approximate DP have also been analyzed recently in \cite{awan2023privacy}.
%while \cite{canonne2020discrete} analyze a method to efficiently sample exactly from a discrete Gaussian. They also point out that truncating the sampler results in inflating the $\delta$ parameter for an $(\varepsilon, \delta)$-DP guarantee. \cite{seeman2021exact} truncate a Markov Chain, which results in total variation distance approximations.
\end{remark}

\begin{algorithm}[h]
\caption{Certified truncated rejection sampler}
\label{alg:truncated-certified-sampler}
\begin{algorithmic}[1]
\STATE Input graph $A$, candidate $\sigma^\star\in\Sigma$, truncation level
$M\ge 1$.
\FOR{$t = 1, \dots, M$}
\STATE For each vertex $i$, independently propose a label $\sigma_i$ by
\[
\Pp(\sigma_i=\sigma_i^\star)=\frac{1}{1+(K-1)e^{-\kappa}},
\qquad
\Pp(\sigma_i=\ell)=\frac{e^{-\kappa}}{1+(K-1)e^{-\kappa}}
\quad(\ell\ne\sigma_i^\star).
\]
\STATE If $\sigma\notin\Sigma$ or $\sigma\notin R_{\sigma^\star}$, reject this proposal and continue to the next iteration.
\STATE Accept $\sigma$ with probability  $b_A(\sigma)
        :=
        \exp\left\{
        \eta[\Tin_{A,D}(\sigma)-\Tin_{A,D}(\sigma^\star)]
        +\kappa h(\sigma,\sigma^\star)
        \right\}$.
If rejected, continue to the next iteration.
\STATE If $\sigma$ is accepted, draw $\Pi\sim \operatorname{Unif}(\mathfrak S_K)$ and return $\Pi\circ\sigma$.
\ENDFOR
\STATE If no proposal is accepted, draw $\Pi\sim{\rm Unif}(\Sn_K)$ and output
$\Pi \circ \sigma^\star$.
\end{algorithmic}
\end{algorithm}

%%%%%

%\subsection{Privacy}

Let $Q_A^{(M)}$ denote the output law of the truncated sampler (Algorithm~\ref{AlgOverview}, with Algorithm~\ref{alg:truncated-certified-sampler}
used in place of Algorithm \ref{alg:certified-sampler}).
Define $\overline r_M:=\bigl(1-e^{-1/4}\bigr)^M$ and $\xi_M :=
    \max\left\{
        -\log(1-\overline r_M),
        \log\bigl(1+(e^{1/4}-1)\overline r_M\bigr)
    \right\}$.
The first result, proved in Appendix~\ref{AppTrunComp}, shows that the truncated rejection sampler leads to an additive increase in the privacy parameter by at most $2\xi_M$:

\begin{lemma}
\label{lem:truncated-uniform-comparison}
Assume that $\kappa=2\eta\theta/K$ and
$\kappa\ge \log(4n(K-1))$. For every input graph $A$ and every $\sigma\in\Sigma$, we have $e^{-\xi_M}\pi_A(\sigma)
    \le
    Q_A^{(M)}(\sigma)
    \le
    e^{\xi_M}\pi_A(\sigma)$.
In particular, if $\eta=\frac{\varepsilon}{2D}$, the mechanism which samples from $Q_A^{(M)}$ is pure $(\varepsilon+2\xi_M)$-node-DP.
\end{lemma}

Importantly, a utility guarantee for the truncated sampler follows immediately from the utility guarantee of Theorem~\ref{thm:risk}, with an extra factor of $e^{\xi_M}$. The following result is proved in Appendix~\ref{AppTruncUtil}:

\begin{theorem}
\label{thm:truncated-utility}
Suppose Assumptions \ref{ass:mild-K} and  \ref{ass:high-signal} hold, and $K \ge 2$. Let $D=C_{deg}\frac aK$ and $\theta=\chi a$, where $C_{deg}$ is sufficiently large and $\chi>0$ is sufficiently small depending only on $\rho_+$. Let $\widehat\sigma^{(M)}\sim Q_A^{(M)}$ be the output of the truncated node-private mechanism with $\eta=\frac{\varepsilon}{2D}$. There exist constants $L,c_1,c_2,c_3>0$ such that if $\varepsilon\ge L\log(nK)$, the output $\widehat\sigma = \widehat\sigma^{(M)}$ of Algorithm~\ref{AlgOverview}, with Algorithm~\ref{alg:truncated-certified-sampler}
used in place of Algorithm \ref{alg:certified-sampler}, satisfies
\[
    \sup_{\sigma_0\in\Sigma}
    \mathbb E_{\sigma_0} r(\sigma_0,\widehat\sigma^{(M)})
    \le
    e^{\xi_M}
    \left[
        \exp\left\{-c_1\frac{nI}{K}\right\}
        +
        \frac{1}{nK}e^{-c_2\varepsilon}
        +
        (nK)^{-c_3A_0}
    \right].
\]
\end{theorem}

The following corollary, proved in Appendix~\ref{AppCorChoice}, shows that taking $M = O(\mathrm{poly}(n))$ leads to similar utility guarantees as the non-truncated mechanism, with only an $e^{-\Theta(M)}$ increase in privacy.

\begin{corollary}
\label{cor:choice_M_pure}
Let $M = O(\mathrm{poly}(n))$.
\begin{enumerate}
    \item For $\eta=\frac{\varepsilon}{2D}$, the mechanism which samples from $Q_A^{(M)}$ is pure $O(\varepsilon + e^{-\Theta(M)})$-node-private.
%In particular, for $\varepsilon = \Omega(1)$, we obtain pure $O(\varepsilon)$-node-DP.
    \item For the parameter choices of Theorem \ref{thm:runtime}, there exist constants $c,c'>0$ such that, under the SBM input, with
probability at least $1-n^{-c}-(nK)^{-c'A_0}$, the runtime of the truncated mechanism is polynomial in $n$, uniformly over its internal randomness.
    \item If, in addition, the hypotheses of Theorem \ref{thm:risk} hold, then for $K \ge 2$, we have
\[
    \sup_{\sigma_0\in\Sigma}
    \mathbb E_{\sigma_0} r(\sigma_0,\widehat\sigma^{(M)})
    \le
        \exp\left\{-c_4\frac{nI}{K}\right\}
        +
        \frac{1}{nK}e^{-c_5\varepsilon}
        +
        (nK)^{-c_6A_0}.
\]
%for absolute constants $c_4, c_5, c_6 > 0$.
\end{enumerate}
\end{corollary}

Thus, taking $M\asymp \log(n)$, the truncated sampler has worst-case polynomial runtime, privacy parameter $\varepsilon+o(1)$, and the same minimax risk
as in Theorem~\ref{thm:risk}. Consequently, combined with the lower bound \cite[Theorem 4.1]{klopp2026node}, Corollary \ref{cor:choice_M_pure} shows that for $a \asymp K\log(n)$, any fixed $c>0$, and $A_0$ sufficiently large depending on $c$, a privacy level of $\varepsilon=\Theta(\log(n))$ is necessary and sufficient to drive the expected misclassification minimax risk $\mathop{\inf}\limits_{\widehat\sigma:\varepsilon\text{-node-DP}}
\mathop{\sup}\limits_{\sigma_0\in\Sigma}
\mathbb E r(\sigma_0,\widehat\sigma)$ below $n^{-(1+c)}$.

\section{Discussion}

We have presented polynomial-time algorithms for node-private community estimation in SBMs that achieve the minimax rate of \cite{klopp2026node}. Our results are applicable when $K$ grows as $\widetilde{O}(\log(n))$, in which case we show that $\varepsilon \asymp \log(n)$ is necessary and sufficient for a misclassification risk of $\frac{1}{\text{poly}(n)}$.

Interesting open directions include developing and analyzing node-private algorithms (either polynomial- or exponential-time) for exact recovery when $K$ grows more rapidly as a function of $n$, or developing polynomial-time algorithms that are more widely applicable to stochastic block models with \emph{approximately} equal-sized communities. Another important question is to exactly characterize upper and lower bounds for $\varepsilon$ in terms of $n$ and $K$ order for the risk to scale as a more general function $f(n)$. Similar questions could be posed concerning necessary and sufficient scaling of the joint parameters $(\varepsilon, \delta)$ in the setting of approximate differential privacy.

%%%%%

\section{AI declaration}

This paper was written with the help of ChatGPT 5.5 Plus. Although ChatGPT played a critical role in brainstorming the initial idea for a polynomial-time algorithm and providing proof outlines for the main results, the authors were heavily involved in interpreting, correcting, and rigorizing all the arguments in the paper.

%\textcolor{red}{SODA rules: ``AI tools may be used for limited language editing without acknowledgment. Any other use of AI must be clearly disclosed, including the tool used, how it was used, and what content was generated or modified."}

\begin{comment}
\begin{center}
\begin{tabular}{l}
In this note, we answer the call \\
posed by Klopp and Zadik, once open to all: \\
a polynomial-time scheme for private recovery, \\
near their exponential-time discovery. \\
Our method extends, with explicit precision, \\
the penalized likelihood score by Lipschitz decision; \\
then samples labels by accept and reject, \\
with polynomial runtime and privacy kept checked. \\
These tools may travel beyond this case, \\
to other models in the private space. \\
And though ChatGPT-5.5 helped light the flame, \\
the authors revised, redirected, and framed \\
the arguments here with mathematical care, \\
until the final proof took shape from there.
\end{tabular}
\end{center}
\end{comment}

%%%%%

\appendix

\section{Notation}
\label{AppNotation}

For $K \in \mathbb{N}$, we write $[K]$ to denote the set $\{1, \dots, K\}$. Let $\Sn_K$ denote the permutation group of $[K]$, and define the orbit $[\sigma]:=\{\pi\circ\sigma:\pi\in\Sn_K\}$.

For a matrix $A$, we write $\|A\|_1$ to denote the entrywise $\ell_1$-norm. For square matrices $A$ and $B$, we denote their inner product by $\langle A,B\rangle = \Tr(AB^T)$. We write $E(A)$ to denote the set of unordered edges of an adjacency matrix $A$, and we use both $ij$ and $\{i,j\}$ to denote elements of $E(A)$.

For functions $f(n)$ and $g(n)$, we write $f(n) \precsim g(n)$ and $f(n) = O(g(n))$ to mean that $f(n) \le c_1g(n)$, if $n \geq c_2$, for universal constants $c_1, c_2 \in (0, \infty)$, and define $f(n) \succsim g(n)$ and $f(n) = \Omega(g(n))$ analogously. We write $f(n) \asymp g(n)$ and $f(n) = \Theta(g(n))$ when both $f(n) \precsim g(n)$ and $f(n) \succsim g(n)$ hold simultaneously. We write $f(n) = \mathrm{poly}(n)$ when $f(n) = \Theta(n^\alpha)$ for some constant $\alpha > 0$. We use the notation $c, C, c_i$, etc., to denote absolute positive constants, whose values may change between results in our paper.

%%%%%

\section{Approximate differential privacy}
\label{AppApprox}

In this appendix, we derive privacy, runtime, and utility results for approximate Gibbs sampling based on a slightly simpler truncated rejection sampler. The results should be viewed in comparison to Section~\ref{sec:truncated-sampler}. The algorithm and analysis are easier, but the result is approximate rather than pure differential privacy. For completeness, we then derive lower bounds for exact recovery rates under approximate DP in Appendix~\ref{AppLBApprox}.

%%%%%

\subsection{Total variation-based truncated rejection sampling}
\label{AppTruncTV}

We present a variant of the analysis in Section \ref{sec:truncated-sampler}, which results in an approximate DP guarantee. At a high level, this analysis shows that approximate sampling in terms of TV distance translates directly into an approximate differential privacy guarantee. This may be interesting in its own right: Note that in the analysis of Lemma~\ref{lem:truncated-uniform-comparison_approx_DP}, the distribution of the fallback sampling distribution $R_A$ is irrelevant, in contrast to the proof of Lemma~\ref{lem:truncated-certified-decomposition}.

%\subsection{Privacy and runtime}

We begin with a privacy guarantee:

\begin{lemma}
\label{lem:truncated-uniform-comparison_approx_DP}
For every input graph $A$, we have $\|Q_A^{(M)}-\pi_A\|_{\rm TV} \le \overline r_M$. In particular, for $\eta=\frac{\varepsilon}{2D}$, the
truncated mechanism which samples from $Q_A^{(M)}$ is $\left(\varepsilon,(1+e^\varepsilon)\overline r_M\right)$-node-DP.
\end{lemma}

\begin{proof}
If ${\mathsf C}(A)=1$, by the same argument as in
Lemma~\ref{lem:truncated-certified-decomposition} (with $R_A$ arbitrary), we have
\begin{equation}
\label{EqnQAM}
Q_A^{(M)} = (1-r^A_{\rm acc})\pi_A+r^A_{\rm acc} R_A,
\end{equation}
with $r^A_{\rm acc}\le \overline r_M$. If ${\mathsf C}(A)=0$, then $Q_A^{(M)}=\pi_A$, so equation~\eqref{EqnQAM} still holds.

Hence, we have
\[
\|Q_A^{(M)}-\pi_A\|_{\rm TV} =r^A_{\rm acc}\|R_A-\pi_A\|_{\rm TV} \leq r^A_{\rm acc} \leq \overline r_M.
\]
This proves the TV bound.

Now let $A\sim_v A'$, and let $S\subseteq\Sigma$. Since $\eta=\varepsilon/(2D)$, by Lemma \ref{lem:global-sensitivity}, we have $\pi_A(S)\le e^\varepsilon \pi_{A'}(S)$. By the TV bound, we have 
\begin{align*}
Q_A^{(M)}(S) \leq \pi_A(S)+\overline r_M \leq e^\varepsilon \pi_{A'}(S)+\overline r_M \leq e^\varepsilon Q_{A'}^{(M)}(S)+(1+e^\varepsilon)\overline r_M .
\end{align*}
This proves the privacy.
\end{proof}

%Note that the high-probability runtime is precisely the statement of Theorem \ref{thm:truncated-runtime}.

%\subsection{Utility}

A utility guarantee for the truncated sampler follows immediately from the utility guarantee of Theorem~\ref{thm:risk}, with an extra additive $\overline r_M $ term:

\begin{theorem}
\label{thm:truncated-utility_approx_DP}
Suppose Assumptions \ref{ass:mild-K} and  \ref{ass:high-signal} hold, and $K \ge 2$. Let $D=C_{deg}\frac aK$ and $\theta=\chi a$, where $C_{deg}$ is sufficiently large and $\chi>0$ is sufficiently small depending only on $\rho_+$. Let $\widehat\sigma^{(M)}\sim Q_A^{(M)}$ be the output of the truncated node-private mechanism with $\eta=\varepsilon/(2D)$. There exist constants $L,c_1,c_2,c_3>0$ such that, if $\varepsilon\ge L\log(nK)$, the output $\widehat\sigma = \widehat\sigma^{(M)}$ of Algorithm~\ref{AlgOverview}, with Algorithm~\ref{alg:truncated-certified-sampler}
used in place of Algorithm \ref{alg:certified-sampler}, satisfies
\[
    \sup_{\sigma_0\in\Sigma}
    \mathbb E_{\sigma_0} r(\sigma_0,\widehat\sigma^{(M)})
    \le
        \exp\left\{-c_1\frac{nI}{K}\right\}
        +
        \frac{1}{nK}e^{-c_2\varepsilon}
        +
        (nK)^{-c_3A_0} + \overline r_M.
\]
\end{theorem}

\begin{proof}
Fix an input graph $A$ and a true labeling $\sigma_0\in\Sigma$.
%If ${\mathsf C}(A)=1$, by the same argument as in
%Lemma~\ref{lem:truncated-certified-decomposition} (with $R_A$ arbitrary), we have 
%
%and $r^A_{\rm acc}\le \overline r_M$. If ${\mathsf C}(A)=0$, then $Q_A^{(M)}=\pi_A$. In this case, set $r_{\mathrm{acc}}^A:=0$; then equation~\eqref{EqnQAM} still holds.
%and let $R_A$ be any probability
%measure on $\Sigma$. Hence, for ${\mathsf C}(A)=0$, we have $Q_A^{(M)} = (1-r^A_{\rm acc})\pi_A+r^A_{\rm acc} R_A$, with $r^A_{\rm acc}\le \overline r_M$. 
Combining the fact that $r(\sigma_0,\sigma) \in [0, 1]$ with equation~\eqref{EqnQAM}, we have
\begin{align*}
\mathbb E_{Q_A^{(M)}} r(\sigma_0,\sigma) = (1 - r^A_{\rm acc})\mathbb E_{\pi_A}r(\sigma_0,\sigma) + r^A_{\rm acc}\mathbb E_{R_A}r(\sigma_0,\sigma) \le \mathbb E_{\pi_A}r(\sigma_0,\sigma) + r^A_{\rm acc} \le \mathbb E_{\pi_A}r(\sigma_0,\sigma) + \overline r_M.
\end{align*}
Taking an expectation over the SBM graph $A$, conditional on $\sigma_0$, yields
$$\mathbb E_{\sigma_0} r(\sigma_0,\widehat\sigma^{(M)}) \le \mathbb E_{\sigma_0}\mathbb E_{\pi_A}r(\sigma_0,\sigma) + \overline r_M.$$
Taking the supremum over $\sigma_0\in\Sigma$ and applying Theorem \ref{thm:risk} to the Gibbs law $\pi_A$ gives the desired bound.
\end{proof}

Finally, the following corollary shows that taking $M = O(\poly(n))$ leads to similar utility guarantees as the non-truncated mechanism, with an approximate DP guarantee instead of pure DP:

\begin{corollary}
\label{cor:choice_M_approx_DP}
Let $M = \frac{\log\bigl((1+e^\varepsilon)/\delta\bigr)}{-\log(1-e^{-1/4})}$, with $\frac{1 + e^{\varepsilon}}{e^{\mathrm{poly}(n)}} \le \delta < 1$.
\begin{enumerate}
    \item For $\eta=\frac{\varepsilon}{2D}$, the
truncated mechanism which samples from $Q_A^{(M)}$ is $(\varepsilon, \delta)$-node-DP.
    \item For the parameter choices of Theorem \ref{thm:runtime}, there exist constants $c,c'>0$ such that, under the SBM input, with
probability at least $1-n^{-c}-(nK)^{-c'A_0}$, the runtime of the truncated mechanism is polynomial in $n$, uniformly over its internal randomness.
    \item If, in addition, the hypotheses of Theorem \ref{thm:risk} hold, then for $K \ge 2$, we have
\[
    \sup_{\sigma_0\in\Sigma}
    \mathbb E_{\sigma_0} r(\sigma_0,\widehat\sigma^{(M)})
    \le
        \exp\left\{-c_1\frac{nI}{K}\right\}
        +
        \frac{1}{nK}e^{-c_2\varepsilon}
        +
        (nK)^{-c_3A_0} + \delta e^{-\varepsilon}.
\]
\end{enumerate}
\end{corollary}

\begin{proof}
The choice of $M$ gives $(1 + e^{\varepsilon})\overline r_M = \delta$, so the first claim follows directly from Lemma \ref{lem:truncated-uniform-comparison_approx_DP}. Next, since $\frac{1 + e^{\varepsilon}}{e^{\mathrm{poly}(n)}} \le \delta < 1$, we have $M = O(\mathrm{poly}(n))$, so the second claim follows as in the proof of Corollary~\ref{cor:choice_M_pure}.

For the final claim, since we assume $\varepsilon \ge L\log(nK)$, we have
\[
    \sup_{\sigma_0\in\Sigma}
    \mathbb E_{\sigma_0} r(\sigma_0,\widehat\sigma^{(M)})
    \le \exp\left\{-c_1\frac{nI}{K}\right\}
        +
        \frac{1}{nK}e^{-c_2\varepsilon}
        +
        (nK)^{-c_3A_0} + \overline r_M,
\]
by Theorem \ref{thm:truncated-utility_approx_DP}.
Noting that $(1 + e^{\varepsilon})\overline r_M = \delta$ completes the proof.
\end{proof}

%%%%%

\subsection{Lower bound}
\label{AppLBApprox}

In this section, we derive an extension of the lower bound argument in \cite{klopp2026node}, adapted to approximate differential privacy. The key modification is the application of group privacy. First, we describe the aspect of group privacy of size two in the lemma below:

\begin{lemma}
\label{lem:approx-group-privacy-two}
Let $\mathcal M:\mathcal G_n\to [K]^n$ be $(\varepsilon,\delta)$-node differentially private. If
$A,A'\in\mathcal G_n$ satisfy $d_v(A,A')\le 2$, then for every measurable
$S\subseteq [K]^n$, we have
\[
\mathbb P(\mathcal M(A)\in S)
\le
 e^{2\varepsilon}\mathbb P(\mathcal M(A')\in S)
+
(1+e^\varepsilon)\delta .
\]
\end{lemma}

\begin{proof}
Since $d_v(A,A')\le2$, there exists $H\in\mathcal G_n$ such that $A\sim_v H$ and $H\sim_v A'$. Applying $(\varepsilon,\delta)$-node DP twice, we obtain
\[
\mathbb P(\mathcal M(A)\in S)
\le
 e^\varepsilon \mathbb P(\mathcal M(H)\in S)+\delta
\]
and
\[
\mathbb P(\mathcal M(H)\in S)
\le
 e^\varepsilon \mathbb P(\mathcal M(A')\in S)+\delta.
\]
Substituting the second inequality into the first gives
\[
\mathbb P(\mathcal M(A)\in S)
\le
 e^{2\varepsilon}\mathbb P(\mathcal M(A')\in S)
+
(1+e^\varepsilon)\delta,
\]
as required.
\end{proof}

\begin{theorem}
\label{thm:approx-node-private-lb}
Fix $K\ge 2$ and assume $\frac{n}{K}\ge 2$. Let $\Theta=\Theta(n,K,a,b)$ be the homogeneous balanced SBM parameter class with $K$ communities, within-community edge probability $a/n$, across-community edge probability $b/n$, and balanced labelings in $\Sigma$.

For each $\theta\in\Theta$, let $\sigma_\theta\in\Sigma$ be a representative
ground-truth labeling, and let $\mathbb P_\theta$ and $\mathbb E_\theta$ denote the probability and
expectation under the SBM law corresponding to $\theta$. Then
\[
\inf_{\mathcal M:(\varepsilon,\delta)\text{-node-DP}}
\sup_{\theta\in\Theta}
\mathbb P_\theta\bigl(r(\sigma_\theta, \mathcal M(A))>0\bigr)
\ge
\left[
\frac{1-(1+e^\varepsilon)\delta}{1+e^{2\varepsilon}}
\right]_+
\]
and
\[
\inf_{\mathcal M:(\varepsilon,\delta)\text{-node-DP}}
\sup_{\theta\in\Theta}
\mathbb E_\theta r(\sigma_\theta, \mathcal M(A))
\ge
\frac1n
\left[
\frac{1-(1+e^\varepsilon)\delta}{1+e^{2\varepsilon}}
\right]_+,
\]
where $[x]_+:=\max\{x,0\}$. 
\end{theorem}

\begin{proof}
Let $\mathcal M:\mathcal G_n\to [K]^n$ be any $(\varepsilon,\delta)$-node differentially private
mechanism. Fix an arbitrary $\theta\in\Theta$, and let $\sigma:=\sigma_\theta\in\Sigma$ be a representative ground-truth labeling. Since $K\ge 2$ and every community has size at least
$n/K\ge 2$, we may choose two vertices $u\neq v$ in different communities. Write
\[
\sigma(u)=k,
\qquad
\sigma(v)=\ell,
\qquad k\neq \ell.
\]
Define a new labeling $\sigma':[n]\to[K]$ by swapping the labels of $u$ and $v$:
\[
\sigma'(u)=\ell,
\qquad
\sigma'(v)=k,
\qquad
\sigma'(i)=\sigma(i)
\quad\text{for all }i\notin\{u,v\}.
\]
Then $\sigma'\in\Sigma$.
Let $\theta'\in\Theta$ be the SBM parameter corresponding to the labeling $\sigma'$. Define the two exact-recovery target events
\[
E_\sigma:=\{\mathcal{M}(A)\in[\sigma]\},
\qquad
E_{\sigma'}:=\{\mathcal{M}(A)\in [\sigma']\},
\]
where $[\sigma] := \{\pi \circ \sigma: \pi \in \Sn_K\}$.

We first show that the two target orbits are disjoint. Suppose, for contradiction, that there
exists $\pi\in \Sn_K$ such that $\sigma'=\pi\circ\sigma$. Since the community $k$ contains at least two vertices, there exists $w\neq u$ such that
$\sigma(w)=k$. Since $w\notin\{u,v\}$, we also have $\sigma'(w)=\sigma(w)=k$. Therefore,
\[
k=\sigma'(w)=\pi(\sigma(w))=\pi(k).
\]
Similarly, because the community $\ell$ contains at least two vertices, there exists $w'\neq v$
such that $\sigma(w')=\ell$. Since $w'\notin\{u,v\}$, we have $\sigma'(w')=\sigma(w')=\ell$,
so
\[
\ell=\sigma'(w')=\pi(\sigma(w'))=\pi(\ell).
\]
Thus, we have $\pi(k)=k$ and $\pi(\ell)=\ell$. But then $\sigma'(u)=\pi(\sigma(u))=\pi(k)=k$, whereas by construction, $\sigma'(u)=\ell\neq k$. This is a contradiction. This implies that $[\sigma]\cap [\sigma']=\varnothing$, so $E_\sigma\cap E_{\sigma'}=\varnothing$.

We now couple the two SBM graph distributions.
Let \(\mu_\sigma\) and \(\mu_{\sigma'}\) denote the respective laws of the SBM adjacency matrices with true community assignments $\sigma$ and $\sigma'$. Let $A \sim \mu_\sigma$ and $A' \sim \mu_{\sigma'}$.
Construct the coupling $(A,A')$, as follows: For every unordered pair $\{i,j\}$ such that
\[
\{i,j\}\cap\{u,v\}=\varnothing,
\]
the labels of both endpoints are the same under $\sigma$ and $\sigma'$. Hence, the corresponding
edge probabilities under the two SBM laws are equal. For these edges, sample $A_{ij}$ once from
this common Bernoulli law and set $A'_{ij}=A_{ij}$. For edges incident to $u$ or $v$, sample the entries of $A$ according to their correct SBM
probabilities under $\sigma$, and sample the entries of $A'$ according to their correct SBM
probabilities under $\sigma'$. This construction has marginals $A\sim\mu_\sigma$ and $A'\sim\mu_{\sigma'}$.
%This construction gives the correct marginal laws:
%\[
%A\sim \mathrm{SBM}(\sigma),
%\qquad
%A'\sim \mathrm{SBM}(\sigma').
%\]
Note that $A$ and $A'$ can differ only on edges incident to $u$ or $v$, so $d_v(A,A')\le 2$ under this coupling.

Let $S\subseteq [K]^n$ be any measurable event. By Lemma~\ref{lem:approx-group-privacy-two},
conditionally on any coupled pair $(A,A')$ with $d_v(A,A')\le 2$, we have
\[
\mathbb P(\mathcal M(A)\in S\mid A,A')
\le
 e^{2\varepsilon}
\mathbb P(\mathcal M(A')\in S\mid A,A')
+
(1+e^\varepsilon)\delta.
\]
Taking an expectation over the coupling gives
\[
\mathbb P_{\theta}(\mathcal M(A)\in S)
\le
 e^{2\varepsilon}
\mathbb P_{\theta'}(\mathcal M(A)\in S)
+
(1+e^\varepsilon)\delta,
\]
where on the right-hand side, $A$ denotes a freshly drawn adjacency matrix from the SBM law
under $\theta'$. Apply this inequality with $S=[\sigma]$. Then
\[
\mathbb P_{\theta}(E_\sigma)
\le
 e^{2\varepsilon}
\mathbb P_{\theta'}(E_\sigma)
+
(1+e^\varepsilon)\delta.
\]
Since $E_\sigma$ and $E_{\sigma'}$ are disjoint, we have $\mathbb P_{\theta'}(E_\sigma)
\le 1-\mathbb P_{\theta'}(E_{\sigma'})$. Now define
\[
\Delta_\sigma
:=
\mathbb P_{\theta}\bigl(r(\sigma, \mathcal M(A))>0\bigr)
=
1-\mathbb P_{\theta}(E_\sigma), \qquad \Delta_{\sigma'}
:=
\mathbb P_{\theta'}\bigl(r(\sigma', \mathcal M(A))>0\bigr)
=
1-\mathbb P_{\theta'}(E_{\sigma'}).
\]
The preceding inequalities imply that $1-\Delta_\sigma
\le
 e^{2\varepsilon}\Delta_{\sigma'}
+
(1+e^\varepsilon)\delta$. Let $\Delta:=\max\{\Delta_\sigma,\Delta_{\sigma'}\}$. Since $\Delta_\sigma\le \Delta$ and $\Delta_{\sigma'}\le \Delta$, we have $1-(1+e^\varepsilon)\delta
\le
(1+e^{2\varepsilon})\Delta$. Hence,
\[
\Delta
\ge
\frac{1-(1+e^\varepsilon)\delta}{1+e^{2\varepsilon}}.
\]
Since $\Delta\ge 0$, this yields
\[
\Delta
\ge
\left[
\frac{1-(1+e^\varepsilon)\delta}{1+e^{2\varepsilon}}
\right]_+.
\]
Since $\theta,\theta'\in\Theta$, we have $\sup_{\theta\in\Theta}
\mathbb P_\theta\bigl(r(\sigma_\theta, \mathcal M(A))>0\bigr)
\ge
\Delta$. Thus, we have
\[
\sup_{\theta\in\Theta}
\mathbb P_\theta\bigl(r(\sigma_\theta, \mathcal M(A))>0\bigr)
\ge
\left[
\frac{1-(1+e^\varepsilon)\delta}{1+e^{2\varepsilon}}
\right]_+.
\]

It remains to prove the expected-mismatch lower bound. For any two labelings $\sigma,\tau:[n]\to[K]$,
if $r(\sigma,\tau)>0$, at least one vertex is misclassified after the best permutation of
labels, implying that
\[
r(\sigma,\tau)
\ge
\frac1n\mathbf 1\{r(\sigma,\tau)>0\}.
\]
Applying this bound under $\theta$ and $\theta'$ yields $\mathbb E_{\theta}r(\sigma, \mathcal M(A))
\ge
\frac1n\Delta_\sigma$ and $\mathbb E_{\theta'}r(\sigma', \mathcal M(A))
\ge
\frac1n\Delta_{\sigma'}$.
Taking the maximum of the two inequalities, we obtain
\[
\max\left\{
\mathbb E_{\theta}r(\sigma, \mathcal M(A)),
\mathbb E_{\theta'}r(\sigma', \mathcal M(A))
\right\}
\ge
\frac1n\Delta.
\]
Since both $\theta$ and $\theta'$ belong to $\Theta$, we have 
\[
\sup_{\theta\in\Theta}
\mathbb E_\theta r(\sigma_\theta, \mathcal M(A))
\ge
\frac1n
\left[
\frac{1-(1+e^\varepsilon)\delta}{1+e^{2\varepsilon}}
\right]_+.
\]
Taking the infimum over all $(\varepsilon,\delta)$-node-DP mechanisms proves the corresponding
minimax lower bounds.
\end{proof}

The following corollary shows that if both $\delta$ and the target expected misclassification risk are polynomially small, the privacy parameter must satisfy $\varepsilon= \Omega(\log(n))$.

\begin{corollary}
\label{cor:approx-dp-risk-necessary}
Let $\gamma_n$ be such that $1-n\gamma_n-\delta>0$. Suppose an $(\varepsilon,\delta)$-node-DP mechanism $\mathcal M:\mathcal{G}_n\to[K]^n$ satisfies
\[
\sup_{\theta\in\Theta}
\mathbb E_\theta r(\sigma_\theta, \mathcal M(A))\le \gamma_n.
\]
Then
\[
\varepsilon
\ge
\frac12
\log\left(
\frac{1-n\gamma_n-\delta}{n\gamma_n+\delta}
\right).
\]
In particular, if $\gamma_n\precsim \frac{1}{n f(n)}$ and $\delta\precsim n^{-c}$,
where $c>0$ and $f(n)\to\infty$, then for all sufficiently large $n$, we have
\[
\varepsilon
\succsim
\min\{\log(f(n)),\, c\log(n)\}.
\]
\end{corollary}

\begin{proof}
By Theorem~\ref{thm:approx-node-private-lb}, we have
\[
\gamma_n
\ge
\frac1n
\left[
\frac{1-(1+e^\varepsilon)\delta}{1+e^{2\varepsilon}}
\right]_+,
\]
implying that $n\gamma_n(1+e^{2\varepsilon})+(1+e^\varepsilon)\delta\ge 1$. Since $e^\varepsilon\le e^{2\varepsilon}$, we obtain $n\gamma_n e^{2\varepsilon}+\delta e^\varepsilon
\le
(n\gamma_n+\delta)e^{2\varepsilon}$, so
\[
(n\gamma_n+\delta)e^{2\varepsilon}
\ge
1-n\gamma_n-\delta.
\]
If $1-n\gamma_n-\delta>0$, then
\[
\varepsilon
\ge
\frac12
\log\left(
\frac{1-n\gamma_n-\delta}{n\gamma_n+\delta}
\right).
\]

If $\gamma_n \precsim 1/(nf(n))$ and $\delta\precsim n^{-c}$, then $n\gamma_n+\delta
\le
C f(n)^{-1}+C n^{-c}$. Since $f(n)\to\infty$ and $n^{-c}\to 0$, we also have $1-n\gamma_n-\delta\ge 1/2$, for all
sufficiently large $n$. This implies that
\[
\varepsilon
\succsim
\log\left(\frac{1}{f(n)^{-1}+n^{-c}}\right)
\asymp 
\min\{\log(f(n)),\, c\log(n)\}.
\]
\end{proof}

Combining Corollaries \ref{cor:choice_M_approx_DP} and \ref{cor:approx-dp-risk-necessary} allows us to draw some useful conclusions about the minimax risk. Applying $c > 0$ and setting $\delta \asymp n^{-c}$, Corollary \ref{cor:approx-dp-risk-necessary} with $f(n) = n^{c}$, implies that any estimator whose expected misclassification minimax risk $\mathop{\inf}\limits_{\widehat\sigma:(\varepsilon, \delta)\text{-node-DP}}
\mathop{\sup}\limits_{\sigma_0\in\Sigma}
\mathbb E r(\sigma_0,\widehat\sigma)$ is less than $n^{-(1 + c)}$ must satisfy $\varepsilon \succsim \log(n)$. Conversely, Corollary \ref{cor:choice_M_approx_DP} shows that for $a \asymp K\log(n)$, for $A_0$ sufficiently large depending on $c$, and for $\varepsilon = L\log(n)$ (so that $\frac{1 + e^{\varepsilon}}{e^{\poly(n)}} \le \delta$), for sufficiently large $L$, the truncated sampler (1) is $(\varepsilon, \delta)$-node-DP, (2) runs in worst-case polynomial time, uniformly over the internal randomness of the sampler (with high-probability over SBM inputs), and (3) has expected misclassification minimax risk at most $n^{-(1 + c)}$. Since Assumption \ref{ass:mild-K} gives $\log(nK) \asymp \log(n)$, a privacy level $\varepsilon \asymp \log(n)$ is therefore both necessary and sufficient for the value of the expected misclassification minimax risk to fall below $n^{-(1 + c)}$.   

%\begin{remark}
%Consider the setting of Corollary \ref{cor:approx-dp-risk-necessary}. For expected mismatch, the two-point argument gives a logarithmic lower bound when
%$\gamma$ is polynomially smaller than $1/n$, for example $\gamma\le n^{-(1+p)}$ with $p>0$.
%If instead $\gamma\asymp 1/(n\log(n))$, this argument gives only
%$\varepsilon=\Omega(\log(\log(n)))$.
%\end{remark}

%%%%%

\section{Zero-concentrated differential privacy}
\label{AppZero}

In this appendix, we prove a zCDP analog of the
private lower bound (Theorem~\ref{thm:approx-node-private-lb}) for stochastic block model community recovery.
The conclusion is that using the notion of zCDP also leads to a lower bound of
\(\rho=\Omega(\log(n))\) for polynomially small exact-recovery failure, meaning the privacy parameter $\rho$ still needs to grow with $n$.

\subsection{Preliminaries}

For probability distributions $P$ and $Q$ defined on a finite output space, recall that the $\alpha$-Renyi divergence is defined by
\[
D_\alpha(P\|Q)
=
\frac{1}{\alpha-1}
\log\left(\sum_y P(y)^\alpha Q(y)^{1-\alpha}\right),
\]
where $\alpha > 1$. Also recall the KL divergence:
\[
D_{\mathrm{KL}}\bigl(P\,\|\,Q\bigr) = \sum_{y} P(y)\log\left(\frac{P(y)}{Q(y)}\right).
\]
%with the usual convention \(D_\alpha(P\|Q)=+\infty\) if \(P\not\ll Q\). 
We begin by defining the notion of a node-zCDP mechanisms on graph inputs, adapted from the standard definitions of zCDP in~\cite{bun2016concentrated}.

\begin{definition}
A randomized mechanism \(\mathcal M:\G_n\to [K]^n\) is \(\rho\)-\emph{node-zCDP} if, for
all \(A\sim_v A'\) and all \(\alpha>1\), we have
\[
D_\alpha\bigl(\Law(\mathcal M(A))\,\|\,\Law(\mathcal M(A'))\bigr)
\le
\alpha\rho,
\]    
where $\Law(\cdot)$ denotes the probability law.
\end{definition}

%\subsection{Auxiliary facts}

Since zCDP satisfies group privacy properties~\cite{bun2016concentrated}, so does node-zCDP:

\begin{lemma}
\label{lem:two-step-group-privacy}
Let \(\mathcal M:\G_n\to[K]^n\) be \(\rho\)-node-zCDP. If 
\(d_v(A,A')\le2\), then for every \(\alpha>1\), we have
\[
D_\alpha\bigl(\Law(\mathcal M(A))\,\|\,\Law(\mathcal M(A'))\bigr)
\le
4\alpha\rho.
\]
Moreover, $D_{\mathrm{KL}}\bigl(\Law(\mathcal M(A))\,\|\,\Law(\mathcal M(A'))\bigr)\le 4\rho$.
\end{lemma}

\begin{proof}
The group privacy of size two claim follows from  \cite[Proposition 5.3]{bun2016concentrated}. Since the output space \([K]^n\) is finite, we have $D_{\mathrm{KL}}(P_0\|P_2) = \lim_{\alpha\rightarrow1}D_\alpha(P_0\|P_2) \le 4\rho$, as required.

\begin{comment}
It is enough to consider a path \(G_0\sim_v G_1\sim_v G_2\). Let $P_i:=\Law(M(G_i))$, for $i=0,1,2$. The zCDP assumption implies the relevant absolute-continuity relations. Let
\[
X:=\frac{dP_0}{dP_1},
\qquad
Y:=\frac{dP_1}{dP_2}.
\]
Fix \(\alpha>1\). Since \(dP_0/dP_2=XY\), we have
\[
\exp\{(\alpha-1)D_\alpha(P_0\|P_2)\}
=
\Ebb_{P_2}\bigl[(XY)^\alpha\bigr]
=
\Ebb_{P_1}\bigl[X^\alpha Y^{\alpha-1}\bigr].
\]
Apply Holder's inequality with conjugate exponents
\[
p=\frac{2\alpha-1}{\alpha},
\qquad
q=\frac{2\alpha-1}{\alpha-1}.
\]
Then
\[
\Ebb_{P_1}\bigl[X^\alpha Y^{\alpha-1}\bigr]
\le
\left(\Ebb_{P_1}X^{2\alpha-1}\right)^{\alpha/(2\alpha-1)}
\left(\Ebb_{P_1}Y^{2\alpha-1}\right)^{(\alpha-1)/(2\alpha-1)}.
\]
Now
\[
\Ebb_{P_1}X^{2\alpha-1}
=
\exp\{(2\alpha-2)D_{2\alpha-1}(P_0\|P_1)\}
\le
\exp\{(2\alpha-2)(2\alpha-1)\rho\},
\]
because \(G_0\sim_v G_1\) and \(M\) is \(\rho\)-node-zCDP. Also,
\[
\Ebb_{P_1}Y^{2\alpha-1}
=
\Ebb_{P_2}Y^{2\alpha}
=
\exp\{(2\alpha-1)D_{2\alpha}(P_1\|P_2)\}
\le
\exp\{(2\alpha-1)(2\alpha)\rho\},
\]
because \(G_1\sim_v G_2\). Combining the last three displays gives
\[
(\alpha-1)D_\alpha(P_0\|P_2)
\le
\frac{\alpha}{2\alpha-1}(2\alpha-2)(2\alpha-1)\rho
+
\frac{\alpha-1}{2\alpha-1}(2\alpha-1)(2\alpha)\rho.
\]
The right-hand side equals
\[
2\alpha(\alpha-1)\rho+2\alpha(\alpha-1)\rho
=
4\alpha(\alpha-1)\rho.
\]
Dividing by \(\alpha-1\) yields
\[
D_\alpha(P_0\|P_2)\le4\alpha\rho.
\]
\end{comment}
\end{proof}

We will use the fact that the KL divergence is convex:

%Now we state the convexity of KL. This is Theorem 2.7.2 in \cite{cover1999elements}, and is proved for mixtures of $2$ distributions, in the discrete case. We give a short proof in the measure-theoretic sense.

\begin{lemma}[Theorem 2.7.2 of \cite{cover1999elements}]
\label{lem:mixture-kl}
Let \(\nu\) be a probability measure on an index space \(\mathcal Z\), and
let \(\{P_z\}_{z\in\mathcal Z}\), \(\{Q_z\}_{z\in\mathcal Z}\) be probability
kernels on a common finite space. Define
\[
P:=\int P_z\,\nu(dz),
\qquad
Q:=\int Q_z\,\nu(dz).
\]
Then
\[
D_{\mathrm{KL}}(P\|Q)
\le
\int D_{\mathrm{KL}}(P_z\|Q_z)\,\nu(dz).
\]
\end{lemma}

%\begin{proof}
%Let \(\widetilde P\) and \(\widetilde Q\) be the joint laws on
%\(\mathcal Z\times\mathcal Y\) defined by
%\[
%\widetilde P(dz,dy)=\nu(dz)P_z(dy),
%\qquad
%\widetilde Q(dz,dy)=\nu(dz)Q_z(dy).
%\]
%Then
%\[
%D_{\mathrm{KL}}(\widetilde P\|\widetilde Q)
%=
%\int D_{\mathrm{KL}}(P_z\|Q_z)\,\nu(dz).
%\]
%The marginals of \(\widetilde P,\widetilde Q\) on \(\mathcal Y\) are \(P,Q\),
%respectively. Since marginalization is a measurable map and KL divergence cannot increase under measurable maps, $D_{\mathrm{KL}}(P\|Q) \le D_{\mathrm{KL}}(\widetilde P\|\widetilde Q)$, which proves the claim.
%\end{proof}

We will also use the Bretagnolle--Huber testing inequality:

\begin{lemma}[Theorem 14.2 of \cite{lattimore2020bandit}]
\label{lem:bh}
For any two probability measures \(P\) and \(Q\) on a common finite space and any event
\(B\), we have
\[
P(B^c)+Q(B)
\ge
\frac12\exp\{-D_{\mathrm{KL}}(P\|Q)\}.
\]
\end{lemma}

\begin{comment}
\begin{proof}
Let \(\mu=P+Q\), and write \(p=dP/d\mu\), \(q=dQ/d\mu\). Then
\[
P(B^c)+Q(B)
=
\int_{B^c}p\,d\mu+\int_B q\,d\mu
\ge
\int \min\{p,q\}\,d\mu.
\]
Set $m:=\int\min\{p,q\}\,d\mu$, and $H:=\int\sqrt{pq}\,d\mu$. By Cauchy--Schwarz,
\[
H^2 \le
\left(\int\min\{p,q\}\,d\mu\right)
\left(\int\max\{p,q\}\,d\mu\right).
\]
Since
\[
\int\max\{p,q\}\,d\mu
=
\int(p+q-\min\{p,q\})\,d\mu
=
2-m
\le2,
\]
we get \(H^2\le2m\), hence \(m\ge H^2/2\). Next, the order-\(1/2\) Renyi divergence satisfies $D_{1/2}(P\|Q)=-2\log H$, and Renyi divergence is monotone in its order, so
\[
-2\log H
=
D_{1/2}(P\|Q)
\le
D_{\mathrm{KL}}(P\|Q).
\]
Therefore \(H^2\ge \exp\{-D_{\mathrm{KL}}(P\|Q)\}\). Combining this with
\(m\ge H^2/2\) proves the desired result.
\end{proof}
\end{comment}

\subsection{Lower bound}

We begin with the statement of the lower bound on the failure probability of $\rho$-node-zCDP algorithms for exact recovery, followed by a corollary which translates the result into a lower bound on $\rho$.

%For \(\sigma\in\Sigma\), let \(A\sim %\mathrm{SBM}(\sigma,a,b)\) mean that the
%upper-triangular entries are independent and
%\[
%\Pbb_\sigma(A_{ij}=1)
%=
%\begin{cases}
%a/n, & \sigma(i)=\sigma(j),\\
%b/n, & \sigma(i)\neq \sigma(j),
%\end{cases}
%\qquad 1\le i<j\le n,
%\]
%where \(0\le b<a\le n\). The exact numerical values of \(a,b\) will not matter for the lower bound. 

\begin{theorem}
\label{thm:zcdp-lower-bound}
Fix $K\ge 2$ and assume $\frac{n}{K}\ge 2.$ Let $\Theta=\Theta(n,K,a,b)$ be the homogeneous balanced SBM parameter class with
$K$ communities, within-community edge probability $a/n$, across-community edge probability
$b/n$, and balanced labelings in $\Sigma$. For each $\theta\in\Theta$, let $\sigma_\theta\in\Sigma$ be a representative ground-truth labeling, and let $\mathbb P_\theta$ and $\mathbb E_\theta$ denote the corresponding SBM probability and expectation.
%Define
%\[
%\delta_\rho(M)
%:=
%\sup_{\sigma\in\Sigma}
%\Pbb_\sigma\bigl(r(\sigma,M(A))>0\bigr),
%\]
%and
%\[
%R_\rho(M)
%:=
%\sup_{\sigma\in\Sigma}
%\Ebb_\sigma\bigl[r(\sigma,M(A))\bigr].
%\]
Then
\[
\inf_{\mathcal M:\,\rho\text{-node-zCDP}}
\sup_{\theta\in\Theta}
\Pbb_\theta\bigl(r(\sigma_\theta, \mathcal M(A))>0\bigr)
\ge
\frac14e^{-4\rho},
\]
and
\[
\inf_{\mathcal M:\,\rho\text{-node-zCDP}}
\sup_{\theta\in\Theta}
\Ebb_\theta\bigl[r(\sigma_\theta, \mathcal M(A))\bigr]
\ge
\frac{1}{4n}e^{-4\rho}.
\]
\end{theorem}

\begin{proof}
Let \(\mathcal M:\G_n\to[K]^n\) be any \(\rho\)-node-zCDP
mechanism. Taking the infimum over $\mathcal M$ at the end will give the desired minimax lower bounds.

First, we construct two disjoint exact-recovery targets. Fix any $\theta \in \Theta$, and let \(\sigma := \sigma_\theta\in\Sigma\) be a representative ground-truth labeling. Since \(K\ge2\) and each community has at least
two vertices, choose vertices \(u,v\in[n]\) in different communities. Write
\[
\sigma(u)=k,\qquad \sigma(v)=\ell,\qquad k\neq \ell.
\]
Define \(\sigma'\in[K]^n\) by swapping the labels of \(u\) and \(v\):
\[
\sigma'(u)=\ell,\qquad
\sigma'(v)=k,\qquad
\sigma'(i)=\sigma(i)\quad\text{for }i\notin\{u,v\}.
\]
The community sizes of \(\sigma'\) are the same as those of \(\sigma\), so
\(\sigma'\in\Sigma\). Let $\theta' \in \Theta$ be the SBM parameter corresponding to the labeling $\sigma'$. We claim that the two label orbits are disjoint: $[\sigma]\cap[\sigma']=\varnothing$. Indeed, suppose not. Then \(\sigma'\in[\sigma]\), so
\(\sigma'=\pi\circ\sigma\) for some \(\pi\in\mathfrak S_K\). Since community \(k\) has at
least two vertices, choose \(w\neq u\) with \(\sigma(w)=k\). Then \(w\neq v\),
so \(\sigma'(w)=\sigma(w)=k\). Therefore, we have
\[
k=\sigma'(w)=\pi(\sigma(w))=\pi(k),
\]
so \(\pi(k)=k\). Similarly, because community \(\ell\) has at least two
vertices, choose \(w'\neq v\) with \(\sigma(w')=\ell\). Then
\(\sigma'(w')=\sigma(w')=\ell\), so \(\pi(\ell)=\ell\). But then
\[
\sigma'(u)=\pi(\sigma(u))=\pi(k)=k,
\]
contradicting the fact that \(\sigma'(u)=\ell\). Thus, the orbits are disjoint. Let $E:=[\sigma]$ and $E':=[\sigma']$.
These are the exact-recovery events in output space for \(\sigma\) and
\(\sigma'\), respectively, and \(E\cap E'=\varnothing\).

Next, let \(\mu_\sigma\) and \(\mu_{\sigma'}\) be the respective laws of the SBM adjacency matrices with true community assignments $\sigma$ and $\sigma'$. We construct a
coupling \(\nu\) of \(A\sim\mu_\sigma\) and \(A'\sim\mu_{\sigma'}\): For every unordered pair \(\{i,j\}\) with
\(\{i,j\}\cap\{u,v\}=\varnothing\), the labels of both endpoints are unchanged
by the swap. Hence, the edge probability for \(\{i,j\}\) is the same under
\(\sigma\) and \(\sigma'\). Sample one Bernoulli variable with this common
probability and set $A_{ij}=A'_{ij}$. For all edges incident to \(u\) or \(v\), sample the entries of \(A\) and
\(A'\) independently over edges, with their correct SBM marginal laws under
\(\sigma\) and \(\sigma'\), respectively. By construction, we have \(A\sim\mu_\sigma\) and \(A'\sim\mu_{\sigma'}\), and \(A\) and \(A'\)
can differ only on edges incident to \(u\) or \(v\). Therefore, we have $d_v(A,A')\le2$.

Using the node-zCDP property of $\mathcal M$, we now bound the KL divergence between the output laws $\Law_{\sigma}(\mathcal M(A))$ and $\Law_{\sigma'}(\mathcal M(A))$. For a fixed graph \(G\), let \(K_G:=\Law(\mathcal M(G))\), where the law is over the internal randomness of \(\mathcal M\). Let
\[
P:=\Law_{\sigma}(\mathcal M(A)),
\qquad
Q:=\Law_{\sigma'}(\mathcal M(A)).
\]
Equivalently, under the coupling \(\nu\), we have
\[
P=\int K_A\,\nu(dA,dA'),
\qquad
Q=\int K_{A'}\,\nu(dA,dA').
\]
For every coupled pair \((A,A')\), we have \(d_v(A,A')\le2\). Hence, by
Lemma~\ref{lem:two-step-group-privacy}, we have $D_{\mathrm{KL}}(K_A\|K_{A'}) \le 4\rho$. Using Lemma~\ref{lem:mixture-kl}, we have
\[
D_{\mathrm{KL}}(P\|Q)
\le
\int D_{\mathrm{KL}}(K_A\|K_{A'})\,\nu(dA,dA')
\le
4\rho.
\]

We now test between the two exact-recovery targets $\Pbb_\theta\bigl(r(\sigma,\mathcal{M}(A))>0\bigr)$ and $\Pbb_{\theta'}\bigl(r(\sigma',\mathcal{M}(A))>0\bigr)$. Define
\[
\delta_\sigma
:=
P(E^c)
=
\Pbb_\theta\bigl(r(\sigma, \mathcal M(A))>0\bigr),
\]
and
\[
\delta_{\sigma'}
:=
Q((E')^c)
=
\Pbb_{\theta'}\bigl(r(\sigma', \mathcal M(A))>0\bigr).
\]
Since \(E\cap E'=\varnothing\), we have $Q(E)\le Q((E')^c)=\delta_{\sigma'}$. Apply Lemma~\ref{lem:bh} with the event \(B=E\). Then
\[
\delta_\sigma+Q(E)
=
P(E^c)+Q(E)
\ge
\frac12\exp\{-D_{\mathrm{KL}}(P\|Q)\}
\ge
\frac12e^{-4\rho}.
\]
Since \(Q(E)\le\delta_{\sigma'}\), this implies $\delta_\sigma+\delta_{\sigma'} \ge \frac12e^{-4\rho}$, so $\max\{\delta_\sigma,\delta_{\sigma'}\} \ge \frac14e^{-4\rho}$. Both \(\sigma\) and \(\sigma'\) belong to \(\Sigma\), so
\begin{equation}
\label{EqnDeltaRho}
\sup_{\theta\in\Theta}
\Pbb_\theta\bigl(r(\sigma_\theta, \mathcal M(A))>0\bigr)
\ge
\frac14e^{-4\rho}.
\end{equation}

It remains to prove the expected-mismatch lower bound. For every \(\tau\in\Sigma\), we have
%\(n \cdot r(\tau,M(A)) \in \{0,1,\ldots,n\}\). Hence, we have
\[
r(\tau, \mathcal M(A))
\ge
\frac1n\,\1\{r(\tau, \mathcal M(A))>0\}.
\]
Taking an expectation and then a supremum over \(\theta\in\Theta\), and using inequality~\eqref{EqnDeltaRho}, we obtain
\[
\sup_{\theta\in\Theta}
\Ebb_\theta r(\sigma_\theta, \mathcal M(A))
\ge
\frac1n
\sup_{\theta\in\Theta}
\Pbb_\theta(r(\sigma_\theta, \mathcal M(A))>0)
\ge
\frac{1}{4n}e^{-4\rho}.
\]
Taking the infimum over all \(\rho\)-node-zCDP mechanisms then proves the corresponding minimax lower bounds.
\end{proof}

The following corollary shows that polynomially small exact-recovery failure or expected misclassification risk requires $\rho=\Omega(\log(n))$:

\begin{corollary}
\label{cor:rho-log-n}
Fix \(c>0\). Under the assumptions of Theorem~\ref{thm:zcdp-lower-bound}:

\begin{enumerate}
\item If a $\rho$-node-zCDP mechanism $\mathcal M:\mathcal{G}_n\to[K]^n$ satisfies $\sup_{\theta\in\Theta}
\Pbb_\theta(r(\sigma_\theta, \mathcal M(A))>0) \le n^{-c}$, then $$\rho \ge \frac{c}{4}\log(n)-\frac14\log(4).$$

\item If a $\rho$-node-zCDP mechanism $\mathcal M:\mathcal{G}_n\to[K]^n$ satisfies $\sup_{\theta\in\Theta}
\Ebb_\theta r(\sigma_\theta, \mathcal M(A)) \le n^{-(1+c)}$, then $$\rho \ge \frac{c}{4}\log(n)-\frac14\log(4).$$
\end{enumerate}
\end{corollary}

\begin{proof}
For exact recovery, Theorem~\ref{thm:zcdp-lower-bound} gives $e^{-4\rho}\le4n^{-c}$. Taking logarithms, we obtain $-4\rho\le \log4-c\log(n)$, from which the first claim follows.

For expected misclassification, Theorem~\ref{thm:zcdp-lower-bound} gives $n^{-(1+c)} \ge \frac{1}{4n}e^{-4\rho}$. Multiplying by \(n\), we obtain $n^{-c} \ge \frac14e^{-4\rho}$, and the same calculation yields the same lower bound on \(\rho\).
\end{proof}

%%%%%

\section{Proofs for Section~\ref{SecScore}}

\subsection{Proof of Lemma~\ref{lem:truth-agreement}}
\label{AppLemTruth}

First, choose $q_i=0$ for all $i$ and $z_{ij}=1$ for all $ij\in E(A)$.  This is in the feasible set~\eqref{eq:H-constraints}; hence, for every $Y$, we have
\begin{equation}
\label{EqnHAD}
\Hin_{A,D}(Y)\le \sum_{ij\in E(A)}Y_{ij}.
\end{equation}
Taking $Y=Y^\sigma$ and subtracting the same penalty gives inequality~\eqref{eq:one-sided}.

Now suppose $\cE_D$ holds and take any feasible $(q,z)$ in the set~\eqref{eq:H-constraints}.  Since $z_{ij}\ge 1-q_i-q_j$ and $Y^{\sigma_0}_{ij}\ge 0$, we have
\begin{align*}
\sum_{ij\in E(A)}z_{ij}Y^{\sigma_0}_{ij}+D\sum_i q_i
&\ge
\sum_{ij\in E(A)}(1-q_i-q_j)Y^{\sigma_0}_{ij}+D\sum_iq_i\\
&=
\sum_{ij\in E(A)}Y^{\sigma_0}_{ij}
+
\sum_i q_i\left(D-
\sum_{j\ne i}A_{ij}\ind\{\sigma_0(j)=\sigma_0(i)\}
\right) \ge
\sum_{ij\in E(A)}Y^{\sigma_0}_{ij}.
\end{align*}
Taking the minimum over feasible $(q,z)$ gives $\Hin_{A,D}(Y^{\sigma_0})\ge \sum_{ij\in E(A)}Y^{\sigma_0}_{ij}$. Combining with inequality~\eqref{EqnHAD}, this is actually an equality. Subtracting the penalty gives $\Tin_{A,D}(\sigma_0)=T_A(\sigma_0)$. Finally, note that if $\Tin_{A,D}(\sigma)\ge\max_{\tau\in\Sigma}\Tin_{A,D}(\tau)-u$, we have
\[
T_A(\sigma)
\ge
\Tin_{A,D}(\sigma)
\ge
\max_\tau \Tin_{A,D}(\tau)-u
\ge
\Tin_{A,D}(\sigma_0)-u
=
T_A(\sigma_0)-u.
\]

%%%%%

\subsection{Proof of Lemma~\ref{lem:global-sensitivity}}
\label{AppLemGlobal}

It suffices to prove the claim for node-adjacent graphs. Suppose $A$ and $A'$ differ only in edges incident to a vertex $v$.  Fix $Y\in[0,1]^{n\times n}$ and let $(q,z)$ be feasible for $\Hin_{A,D}(Y)$.  Construct $(q',z')$ for $A'$, as follows:  Set $q'_v=1$ and $q'_i=q_i$ for $i\ne v$.  For every edge $ij\in E(A')$ not incident to $v$, set $z'_{ij}=z_{ij}$; such an edge also belongs to $E(A)$.  For every edge $vj\in E(A')$ incident to $v$, set $z'_{vj}=0$.

The constraints are feasible.  Nonincident constraints are inherited.  If $vj\in E(A')$, then
\[
        1-q'_v-q'_j=1-1-q'_j=-q'_j\le 0=z'_{vj}.
\]
Thus, $(q',z')$ is feasible for $\Hin_{A',D}(Y)$.  The edge terms corresponding to new incident edges have $z'=0$, nonincident terms are unchanged, and the only possible increase in the deletion penalty is
\[
        D(q'_v-q_v)=D(1-q_v)\le D.
\]
Hence, we have $\Hin_{A',D}(Y)\le \Hin_{A,D}(Y)+D$, implying that $|\Hin_{A,D}(Y)-\Hin_{A',D}(Y)|\le D$, by symmetry. The penalty $-\lambda\sum_{i<j}Y_{ij}$ is independent of $A$. Plugging in $Y=Y^\sigma$ proves the node-adjacent case, and telescoping over a shortest node path proves the claim.

The privacy claim follows immediately from privacy of the exponential mechanism.

%%%%%

\section{Proofs for Section~\ref{SecSampling}}

\subsection{Proof of Lemma~\ref{lem:eroded-sdp}}
\label{AppLemEroded}

Note that for $Y\in\cC$, we have
\[
0 \le Y_{ij} \le \sqrt{Y_{ii} Y_{jj}} = 1, \qquad \forall i,j.
\]
Since $Z_{ij}\in\{0,1\}$, we can write
\begin{equation}
\|Y-Z\|_1
=2\sum_{i<j:Z_{ij}=1}(1-Y_{ij})
+2\sum_{i<j:Z_{ij}=0}Y_{ij},
\label{eq:l1-affine}
\end{equation}
which is affine in $Y$ over $\cC$.

It remains to represent the hypograph of $\Hin_{A,D}$.  For fixed $Y$, the LP \eqref{eq:H-tilde}--\eqref{eq:H-constraints} is feasible and bounded.  Introduce nonnegative dual variables $\alpha_i\ge0$, $\beta_{ij}\ge0$, and $\zeta_{ij}\ge0$ for the constraints
\[
        q_i\le 1,
        \qquad
        z_{ij}\le 1,
        \qquad
        1-q_i-q_j-z_{ij}\le 0.
\]
The Lagrangian is
\begin{align*}
L(q,z;\alpha,\beta,\zeta)
&=
\sum_{ij\in E(A)}z_{ij}Y_{ij}+D\sum_iq_i
+\sum_i\alpha_i(q_i-1)
+\sum_{ij\in E(A)}\beta_{ij}(z_{ij}-1) +\sum_{ij\in E(A)}\zeta_{ij}(1-q_i-q_j-z_{ij})\\
&=
\sum_i\left(D+\alpha_i-\sum_{j:ij\in E(A)}\zeta_{ij}\right)q_i
+\sum_{ij\in E(A)}(Y_{ij}+\beta_{ij}-\zeta_{ij})z_{ij}\\
&\qquad
+\sum_{ij\in E(A)}\zeta_{ij}-\sum_i\alpha_i-\sum_{ij\in E(A)}\beta_{ij}.
\end{align*}
Taking the infimum over $q,z\ge0$ is finite if and only if
\begin{equation}
        D+\alpha_i-\sum_{j:ij\in E(A)}\zeta_{ij}\ge0,
        \qquad
        Y_{ij}+\beta_{ij}-\zeta_{ij}\ge0.
\label{eq:dual-constraints}
\end{equation}
Thus, strong duality gives
\begin{equation*}
\Hin_{A,D}(Y)
=
\max_{\alpha,\beta,\zeta\ge0}
\left\{
\sum_{ij\in E(A)}\zeta_{ij}-\sum_i\alpha_i-\sum_{ij\in E(A)}\beta_{ij}
\right\},
\end{equation*}
subject to the constraints~\eqref{eq:dual-constraints}. Consequently, $t\le\Hin_{A,D}(Y)$ if and only if there exist nonnegative $\alpha,\beta,\zeta$ satisfying the constraints~\eqref{eq:dual-constraints} and
\begin{equation}
        t\le
        \sum_{ij\in E(A)}\zeta_{ij}-\sum_i\alpha_i-\sum_{ij\in E(A)}\beta_{ij}.
\label{eq:hypograph-linear}
\end{equation}
Therefore, using equation~\eqref{eq:l1-affine}, we see that $V_\theta(Z)$ is the value of the SDP that maximizes
\[
        t-\lambda\sum_{i<j}Y_{ij}+\frac\theta n\|Y-Z\|_1
\]
over variables $(Y,t,\alpha,\beta,\zeta)$, subject to $Y\in\cC$, inequalities~\eqref{eq:dual-constraints} and~\eqref{eq:hypograph-linear}, and nonnegativity of the dual variables.  This formulation has polynomially many variables and constraints, and the only nonlinear constraint is $Y\succeq0$.

%%%%%

\subsection{Proof of Lemma~\ref{lem:hungarian-canon}}
\label{AppLemRomanian}

We first recall a standard formulation of the linear sum assignment problem: Given a weight matrix \(W=(W_{ab})\in \mathbb{R}^{K\times K}\), the maximum-weight
assignment problem is $\max_{\pi\in \Sn_K} \sum_{a=1}^K W_{a,\pi(a)}$. Equivalently, writing \(x_{ab}\in\{0,1\}\) for whether row \(a\) is assigned to
column \(b\), this is the integer program
\begin{align}
\label{eq:Max_Weight_Assign}
    \max_{x\in\{0,1\}^{K\times K}}
        \quad & \sum_{a=1}^K\sum_{b=1}^K W_{ab}x_{ab}\notag \\
    \text{subject to}
        \quad & \sum_{b=1}^K x_{ab}=1, \qquad a=1,\dots,K,\\
        \quad & \sum_{a=1}^K x_{ab}=1, \qquad b=1,\dots,K\notag.
\end{align}
The feasible binary matrices are exactly the permutation matrices, so the two
formulations are identical. The Hungarian algorithm, or Kuhn--Munkres algorithm,
solves the linear sum assignment problem in \(O(K^3)\) arithmetic operations \cite[Chapter 11.2]{papadimitriou1998combinatorial}.

%Note that the set
%\[
%        [\sigma]=\{\pi\circ\sigma:\pi\in\mathfrak S_K\}
%\]
%is finite and nonempty.  Therefore the closest elements of \([\sigma]\) to
%\(\sigma^\star\) exist, and lexicographic tie-breaking selects a unique one.  Thus
%\(\operatorname{Can}_{\sigma^\star}(\sigma)\) is well-defined and belongs to
%\([\sigma]\).
%First note that $R_{\sigma^*}$ contains exactly one representative from every orbit $[\sigma] \subseteq \Sigma$; uniqueness follows from lexicographic tie-breaking.

Now note that if \(\rho\in\mathfrak S_K\), we have
\[
        \{\pi\circ(\rho\circ\sigma):\pi\in\mathfrak S_K\}
        =
        \{\pi\circ\sigma:\pi\in\mathfrak S_K\}.
\]
Hence, \(\operatorname{Can}_{\sigma^\star}(\rho\circ\sigma)
=\operatorname{Can}_{\sigma^\star}(\sigma)\).  The canonical representative depends only
on the orbit, not on the particular labeling used to describe it. Therefore, every orbit
contains exactly one element of \(R_{\sigma^\star}\), namely its canonical
representative. If \(\sigma\in R_{\sigma^\star}\), then \(\sigma\) is the selected closest
element of its orbit to \(\sigma^\star\).  Hence, we have
\[
        h(\sigma,\sigma^\star)
        =
        \min_{\pi\in\mathfrak S_K}h(\pi\circ\sigma,\sigma^\star)
        =
        \dorb([\sigma],[\sigma^\star]).
\]

It remains to prove polynomial-time computability. Define the confusion matrix
\[
        N_{ab}(\sigma,\sigma^\star)
        :=
        |\{i:\sigma_i=a,\ \sigma_i^\star=b\}|,
        \qquad a,b\in[K].
\]
For a permutation \(\pi\), the number of agreements between \(\pi\circ\sigma\) and
\(\sigma^\star\) is
\[
        A(\pi):=\sum_{a=1}^K N_{a,\pi(a)}(\sigma,\sigma^\star).
\]
Indeed, vertices with \(\sigma_i=a\) receive label \(\pi(a)\) after relabeling, and
exactly \(N_{a,\pi(a)}(\sigma,\sigma^\star)\) of them have
\(\sigma_i^\star=\pi(a)\). Therefore, we have
\[
        h(\pi\circ\sigma,\sigma^\star)=n-A(\pi),
\]
so closest relabelings are exactly the maximizers of \(A(\pi)\). We now encode the lexicographic tie-breaking into the same assignment problem. Set
\[
        B:=2K+1,
        \qquad
        L_a(\sigma):=\sum_{i:\sigma_i=a}B^{n-i},
\]
and define
\[
        W_{ab}:=B^{n+1}N_{ab}(\sigma,\sigma^\star)-bL_a(\sigma).
\]
For a permutation \(\pi\), let
\[
        F(\pi):=\sum_{a=1}^K W_{a,\pi(a)}
        =
        B^{n+1}A(\pi)-R(\pi),
        \qquad
        R(\pi):=\sum_{a=1}^K\pi(a)L_a(\sigma).
\]

We claim that a maximizer of \(F\) gives exactly
\(\operatorname{Can}_{\sigma^\star}(\sigma)\). We first prove that every maximizer of \(F\) maximizes \(A\).  For every
\(\pi\in\mathfrak S_K\), since $B=2K+1$, we have 
\[
        0\le R(\pi)
        =
        \sum_{i=1}^n(\pi\circ\sigma)_i B^{n-i}
        \le
        K\sum_{r=0}^{n-1}B^r
        <
        B^n,
\]
where the identity follows from
\[
        \sum_{a=1}^K \pi(a)L_a(\sigma)
        =
        \sum_{a=1}^K \pi(a)\sum_{i:\sigma_i=a}B^{n-i}
        =
        \sum_{i=1}^n(\pi\circ\sigma)_iB^{n-i}.
\]
If \(A(\pi)\ge A(\rho)+1\), then
\[
\begin{aligned}
        F(\pi)-F(\rho) =
        B^{n+1}\{A(\pi)-A(\rho)\}
        -\{R(\pi)-R(\rho)\} \ge
        B^{n+1}-R(\pi)+R(\rho)
        >
        B^{n+1}-B^n
        >0.
\end{aligned}
\]
Hence, \(A(\pi)>A(\rho)\) implies \(F(\pi)>F(\rho)\). Therefore, any maximizer of
\(F\) must belong to \(\arg\max_{\pi}A(\pi)\). Now restrict to $\mathcal M:=\arg\max_{\pi\in\mathfrak S_K}A(\pi)$. For \(\pi\in\mathcal M\), note that \(A(\pi)\) is constant. Hence, maximizing \(F(\pi)\) over
\(\mathcal M\) is equivalent to minimizing \(R(\pi)\) over \(\mathcal M\): $$\arg\max_{\pi\in\mathcal M}F(\pi)=\arg\min_{\pi\in\mathcal M}R(\pi).$$ Finally, we have
\[
        R(\pi)
        =
        \sum_{i=1}^n(\pi\circ\sigma)_iB^{n-i}.
\]
Since \(B>K\), minimizing this base-\(B\) encoding selects the lexicographically
smallest relabeled vector \(\pi\circ\sigma\). Hence, a maximizer of \(F\) gives exactly
\(\operatorname{Can}_{\sigma^\star}(\sigma)\). 

Regarding exact computation, maximizing $F(\pi)=\sum_{a=1}^K W_{a,\pi(a)}$ is precisely the standard maximum-weight assignment problem~\eqref{eq:Max_Weight_Assign} on the \(K\times K\) matrix
\(W=(W_{ab})\). This can be solved by the Hungarian algorithm in \(O(K^3)\) arithmetic operations. Since \(B=2K+1\), \(0\le N_{ab}\le n\), and
\(0\le L_a(\sigma)<B^n\), the integer weights \(W_{ab}\) have polynomial encoding length in \(n\). After computing the maximizing permutation \(\pi_{\mathrm H}\), we have
\[
        \operatorname{Can}_{\sigma^\star}(\sigma)
        =
        \pi_{\mathrm H}\circ\sigma.
\]
Thus, membership in \(R_{\sigma^\star}\) is tested by computing
\(\pi_{\mathrm H}\circ\sigma\), and checking whether it equals \(\sigma\), which takes
\(O(n)\) additional time. We conclude that membership in \(R_{\sigma^\star}\) is testable
in polynomial time using one \(K\times K\) assignment problem.

%%%%%

\subsection{Proof of Lemma~\ref{lem:geometry}}
\label{AppLemGeom}

Relabel $\sigma$ so that $h(\sigma,\tau)=\dorb([\sigma],[\tau])$.  Let
\[
        C_x:=\{i:\tau_i=x\},
        \qquad
        D_y:=\{i:\sigma_i=y\},
\]
and define the normalized confusion matrix $M_{xy}:=\frac{|C_x\cap D_y|}{n}$. Every row and column of $M$ has sum $1/K$.  Optimal alignment gives
\begin{align}
        \sum_x M_{xx}\ge \sum_x M_{x,\pi(x)},
        \qquad\forall \pi\in\Sn_K,
\label{eq:optim_align}
\end{align}
because we have relabeled \(\sigma\) to minimize the Hamming distance to \(\tau\). Now, $KM$ is doubly stochastic. By the Birkhoff--von Neumann theorem~\cite[Theorem 1.1]{chen2019birkhoff}, we have
\[
        KM=\sum_{\pi\in\Sn_K} w_\pi P_\pi,
        \qquad
        w_\pi\ge0,
        \quad
        \sum_\pi w_\pi=1,
\]
where $P_\pi$ is the permutation matrix. Using inequality~\eqref{eq:optim_align} and multiplying by $K$, we have
\[
        \langle KM,P_\pi\rangle\le \langle KM,I\rangle=\Tr(KM),
        \qquad\forall \pi,
\]
implying that
\[
        \|KM\|_F^2
        =\langle KM,KM\rangle
        =\sum_\pi w_\pi\langle KM,P_\pi\rangle
        \le \Tr(KM).
\]
Equivalently, we have $\sum_{x,y}M_{xy}^2\le\frac1K\sum_xM_{xx}$. The number of ordered pairs that are in the same cluster under $\tau$ is $Ks^2=n^2/K$, and similarly for $\sigma$. The number of ordered pairs in the same cluster under both labelings is
\[
        \sum_{x,y}|C_x\cap D_y|^2
        =n^2\sum_{x,y}M_{xy}^2.
\]
Thus, we have
\[
        \frac1{n^2}\|Y^\sigma-Y^\tau\|_1
        =
        \frac2K-2\sum_{x,y}M_{xy}^2
        \ge
        \frac2K\left(1-\sum_xM_{xx}\right).
\]
Finally, we have $1-\sum_xM_{xx} = \frac{h(\sigma,\tau)}{n} = \frac{\dorb([\sigma],[\tau])}{n}$, establishing the claim.

%%%%%

\subsection{Proof of Lemma~\ref{lem:certificate-margin}}
\label{AppCertMar}

The certificate says that, for every $Y\in\cC$, we have
\[
        \Psiin_{A,D}(Y)+\frac\theta n\|Y-Y^{\sigma^\star}\|_1
        \le
        \Psiin_{A,D}(Y^{\sigma^\star}).
\]
Taking $Y=Y^\sigma$ gives
\[
        \Tin_{A,D}(\sigma^\star)-\Tin_{A,D}(\sigma)
        \ge
        \frac\theta n\|Y^\sigma-Y^{\sigma^\star}\|_1 \ge
        \frac{2\theta}{K}\dorb([\sigma],[\sigma^\star]),
\]
as required, by Lemma \ref{lem:geometry}.

For $\sigma\in R_{\sigma^\star}$, Lemma \ref{lem:hungarian-canon} gives $h(\sigma,\sigma^\star)=\dorb([\sigma],[\sigma^\star])$. Thus, inequality~\eqref{eq:certified-margin-theta} gives
\[
        \Tin_{A,D}(\sigma)-\Tin_{A,D}(\sigma^\star)
        \le
        -\frac{2\theta}{K} h(\sigma,\sigma^\star).
\]
Multiplying by $\eta$ and using $\kappa=2\eta\theta/K$, yields
\[
        \eta[\Tin_{A,D}(\sigma)-\Tin_{A,D}(\sigma^\star)]
        +\kappa h(\sigma,\sigma^\star)
        \le 0.
\]
Hence, $b_A(\sigma)\le1$.

%%%%%

\subsection{Proof of Lemma~\ref{lem:sampler-exact}}
\label{AppLemExact}

By Lemma \ref{lem:certificate-margin}, we have $b_A(\sigma)\le 1$ for every $\sigma\in\Sigma\cap R_{\sigma^\star}$.
We now compute the law of the accepted canonical representative. Let $G:=\bigl(1+(K-1)e^{-\kappa}\bigr)^n $. For every $\sigma\in[K]^n$, the product proposal probability is 
\begin{equation}
\nu_A(\sigma)=G^{-1}e^{-\kappa h(\sigma,\sigma^\star)}. \label{eq:proposal-density} 
\end{equation}
Define the one-trial acceptance function 
\[ 
a_A(\sigma) := 
\begin{cases} \exp\left\{ \eta\bigl(\widetilde T_{A,D}(\sigma) -\widetilde T_{A,D}(\sigma^\star)\bigr) + \kappa h(\sigma,\sigma^\star) \right\}, & \sigma\in\Sigma\cap R_{\sigma^\star},\\ 0, & \text{otherwise}. 
\end{cases} 
\]
For $\sigma\in\Sigma\cap R_{\sigma^\star}$, we have $b_A(\sigma) \le 1$, so $0\le a_A(\sigma)\le 1$. The probability that a single proposal trial is accepted is 
\[ 
p^A_{\rm acc} := \sum_{\sigma\in[K]^n} \nu_A(\sigma)a_A(\sigma).
\]
Since $a_A(\sigma)=0$ outside $\Sigma\cap R_{\sigma^\star}$, using equation~\eqref{eq:proposal-density} gives 
\begin{align} p^A_{\rm acc} &= \sum_{\sigma\in\Sigma\cap R_{\sigma^\star}} G^{-1}e^{-\kappa h(\sigma,\sigma^\star)} \exp\left\{ \eta\bigl(\widetilde T_{A,D}(\sigma) -\widetilde T_{A,D}(\sigma^\star)\bigr) + \kappa h(\sigma,\sigma^\star) \right\} \notag\\ 
&= G^{-1}e^{-\eta\widetilde T_{A,D}(\sigma^\star)} \sum_{\sigma\in\Sigma\cap R_{\sigma^\star}} e^{\eta\widetilde T_{A,D}(\sigma)}. 
\label{eq:pacc-formula} 
\end{align}  
In particular, we have $p^A_{\rm acc}>0$, since $\sigma^\star\in\Sigma\cap R_{\sigma^\star}$. Let $T$ be the first accepted trial, and let $\widehat\sigma_R$ denote the accepted canonical representative before the final random relabeling. Fix $\sigma\in\Sigma\cap R_{\sigma^\star}$. Then

\begin{align} \mathbb P(\widehat\sigma_R=\sigma) = \sum_{t=1}^{\infty} \mathbb P(T=t,\widehat\sigma_R=\sigma) = \sum_{t=1}^{\infty} (1-p^A_{\rm acc})^{t-1}\nu_A(\sigma)a_A(\sigma) = \frac{\nu_A(\sigma)a_A(\sigma)}{p^A_{\rm acc}}.
\label{eq:accepted-canonical-law-raw}
\end{align} 

Substituting equation~\eqref{eq:proposal-density}, the definition of $a_A(\sigma)$, and equation~\eqref{eq:pacc-formula} into equation~\eqref{eq:accepted-canonical-law-raw}, the factors involving $G$, $\kappa h(\sigma,\sigma^\star)$, and $\widetilde T_{A,D}(\sigma^\star)$ cancel. We obtain 
\begin{align}
\mathbb P(\widehat\sigma_R=\sigma) = \frac{e^{\eta\widetilde T_{A,D}(\sigma)}} {\sum_{\tau\in\Sigma\cap R_{\sigma^\star}} e^{\eta\widetilde T_{A,D}(\tau)}} , \qquad \sigma\in\Sigma\cap R_{\sigma^\star}.
\label{eq:canonical-gibbs-law} 
\end{align}

It remains to pass from canonical representatives to all labeled assignments in $\Sigma$. Since $\Sigma$ is the exact-size labeling space, each label is used exactly $s=n/K$ times. Therefore, no non-identity permutation in $\mathfrak S_K$ fixes a labeling $\sigma\in\Sigma$. Hence, every orbit
\[ 
[\sigma]=\{\pi\circ\sigma:\pi\in \mathfrak S_K\} 
\]
has exactly $K!$ distinct elements. By Lemma \ref{lem:hungarian-canon}, the set $R_{\sigma^\star}$ contains exactly one representative from each orbit. Consequently, the map $\mathfrak S_K\times R_{\sigma^\star}\to\Sigma$, such that $(\pi,r)\mapsto \pi\circ r$, is a bijection. Fix an arbitrary $\sigma\in\Sigma$. Let $r_\sigma\in R_{\sigma^\star}$ be the unique canonical representative of the orbit $[\sigma]$. By the bijection, there is a unique permutation \(\pi_\sigma\in \mathfrak S_K\) such that $\sigma=\pi_\sigma\circ r_\sigma$. The final output is $\widehat\sigma=\Pi\circ \widehat\sigma_R$, where $\Pi\sim\operatorname{Unif}(\mathfrak S_K)$ is independent of $\widehat\sigma_R$. Thus, we have
\begin{align} 
\mathbb P(\widehat\sigma=\sigma) &= \sum_{\pi\in \mathfrak S_K} \mathbb P\bigl(\Pi=\pi,\ \pi\circ\widehat\sigma_R=\sigma\bigr) = \sum_{\pi\in \mathfrak S_K} \mathbb P(\Pi=\pi)\, \mathbb P\bigl(\widehat\sigma_R=\pi^{-1}\circ\sigma\bigr). 
\label{eq:sum-over-perms}
\end{align}  
For $\pi\ne\pi_\sigma$, the labeling $\pi^{-1}\circ\sigma$ is not the canonical representative $r_\sigma$. Hence, we have $\pi^{-1}\circ\sigma\notin R_{\sigma^\star}$, and since $\widehat\sigma_R\in R_{\sigma^\star}$ almost surely, we have $\mathbb{P}\bigl(\widehat\sigma_R=\pi^{-1}\circ\sigma\bigr)=0$. Therefore, only the term \(\pi=\pi_\sigma\) contributes to equation~\eqref{eq:sum-over-perms}, and $\mathbb P(\widehat\sigma=\sigma) = \frac{1}{K!}\, \mathbb P(\widehat\sigma_R=r_\sigma)$. Using equation~\eqref{eq:canonical-gibbs-law}, we have
\begin{align}
\mathbb P(\widehat\sigma=\sigma) = \frac{1}{K!} \frac{e^{\eta\widetilde T_{A,D}(r_\sigma)}} {\sum_{r\in\Sigma\cap R_{\sigma^\star}} e^{\eta\widetilde T_{A,D}(r)}}. 
\label{eq:almost-final-law} 
\end{align}
Finally, note that $\widetilde T_{A,D}$ is invariant under global relabeling, because it depends on $\sigma$ only through the cluster matrix $Y^\sigma$. Hence, we have $\widetilde T_{A,D}(r_\sigma)=\widetilde T_{A,D}(\sigma)$. Moreover, since each orbit has exactly $K!$ elements and $\widetilde T_{A,D}$ is constant on each orbit, we have
\[ 
\sum_{\tau\in\Sigma}e^{\eta\widetilde T_{A,D}(\tau)} = K! \sum_{r\in\Sigma\cap R_{\sigma^\star}} e^{\eta\widetilde T_{A,D}(r)}. 
\] 
Combining this with equation~\eqref{eq:almost-final-law} gives
\[
\mathbb P(\widehat\sigma=\sigma) = \frac{e^{\eta\widetilde T_{A,D}(\sigma)}} {\sum_{\tau\in\Sigma}e^{\eta\widetilde T_{A,D}(\tau)}}, 
\]
which is the desired Gibbs distribution.

%\begin{lemma}[Expected number of proposal trials]
%\label{lem:runtime-trials}
%If $\kappa\ge \log(4n(K-1))$, then the expected number of proposal trials in Algorithm \ref{alg:certified-sampler} is at most $e^{1/4}$.
%\end{lemma}

Turning to the expected number of proposal trials, note that since $(K-1)e^{-\kappa}\le(4n)^{-1}$, we have
\[
        G=(1+(K-1)e^{-\kappa})^n\le \left(1+\frac1{4n}\right)^n\le e^{1/4}.
\]
Let
\[
        Z_A:=\sum_{\sigma\in\Sigma\cap R_{\sigma^\star}}
        \exp\{\eta[\Tin_{A,D}(\sigma)-\Tin_{A,D}(\sigma^\star)]\}.
\]
The labeling $\sigma^\star$ is canonical and contributes 1, so $Z_A\ge1$.  From the equation \eqref{eq:pacc-formula}, we have $p_{\mathrm{acc}}=G^{-1}Z_A\ge G^{-1}$. Thus, we have $$\Ee[\text{trials}]=1/p_{\mathrm{acc}}\le G\le e^{1/4}.$$

%%%%%

\section{Proofs for Section~\ref{SecSBMAnalysis}}

%%%%%

\subsection{Proof of Lemma~\ref{lem:true-within-conc}}
\label{AppLemTrue}

First, we have a result based on a Chernoff bound for independent Bernoulli variables. It is a special case of \cite[Theorem 1]{hoeffding1963probability}.

\begin{lemma}[Theorem 1 in \cite{hoeffding1963probability}]
\label{lemma:Chern_Bern}
Let \(X_r \overset{indep.}{\sim} \mathrm{Bern}(p_r)\), for $r \in [m]$. Set $X=\sum_{r=1}^m X_r$ and $\mu=\mathbb E X$. For every \(t>\mu\), we have
\[
\mathbb P(X\ge t)
\le
\left(\frac{e\mu}{t}\right)^t.
\]
\end{lemma}

\begin{comment}
\begin{proof}
For any \(\lambda>0\), Markov's inequality gives
\[
\mathbb P(X\ge t)
=
\mathbb P(e^{\lambda X}\ge e^{\lambda t})
\le
e^{-\lambda t}\mathbb E e^{\lambda X}
=
e^{-\lambda t} \prod_{r=1}^m \mathbb E e^{\lambda X_r}.
\]
Since \(X_r\sim \operatorname{Bern}(p_r)\), we have 
\[
\mathbb E e^{\lambda X_r}
=
1-p_r+p_re^\lambda
=
1+p_r(e^\lambda-1)
\le
\exp\{p_r(e^\lambda-1)\}.
\]
Thus, we have
%\[
%\mathbb E e^{\lambda X}
%\le
%\exp\left\{(e^\lambda-1)\sum_{r=1}^m p_r\right\}
%=
%\exp\{\mu(e^\lambda-1)\}.
%\]
%Thus
\[
\mathbb P(X\ge t)
\le
\exp\left\{-\lambda t + (e^\lambda-1)\sum_{r=1}^m p_r\right\} \le 
\exp\{-\lambda t+\mu(e^\lambda-1)\}.
\]
Choose $\lambda=\log(t/\mu)$. Since \(t>\mu\), this choice satisfies \(\lambda>0\), and
\(e^\lambda=t/\mu\). Hence, we have
\[
\mathbb P(X\ge t)
\le
\exp\{-t\log(t/\mu)+t-\mu\}
\le
\exp\{-t\log(t/\mu)+t\}.
\]
Finally, $\exp\{-t\log(t/\mu)+t\} = \left(\frac{e\mu}{t}\right)^t$, proving the claim.
\end{proof}
\end{comment}

For a fixed vertex $i$, define $ X_i:=\sum_{j\ne i}A_{ij}\ind\{\sigma_0(j)=\sigma_0(i)\}$. Then $X_i$ is a sum of $s-1$ independent Bernoulli variables with parameter $a/n$, so
\[
        \mu_i:=\Ee X_i\le \frac{s a}{n}=\frac aK.
\]
For $D=C_{deg}a/K$ with $C_{deg}$ large enough, Lemma \ref{lemma:Chern_Bern} with $t = D$ implies $\Pp(X_i\ge D)\le e^{-cD}$.  A union bound gives
\[
        \Pp(\cE_D^c)
        \le
        n e^{-cC_{deg}a/K}
        \le
        \exp\{\log(n)-cC_{deg} A_0\log(nK)\}.
\]
Increasing $C_{deg}$ and absorbing constants yields the desired result.

%%%%%

\subsection{Proof of Lemma~\ref{lem:eroded-cert-highprob}}
\label{AppErodedCert}

We will use the following result, applicable to weighted adjacency matrices. It follows directly from Theorem 3 of \cite{pirinen2019exact}, specialized to the case of balanced communities with no outlier nodes:

\begin{lemma}
\label{lemma:weighted-sdp}
There is a universal constant $C_{\mathrm{SDP}}>0$ with the following property: Consider a complete weighted graph with independent upper-triangular weights $0\le W_{ij}\le 1$, whose distribution depends only on the labeling $\sigma_0$ of nodes into $K$ equal-sized communities. Let
\begin{align*}
\mu_{\mathrm{in}} & := \Ee W_{ij}\quad (i,j\text{ in the same communities}), \\
\mu_{\mathrm{out}} &:= \Ee W_{ij}\quad (i,j\text{ in different communities}),
\end{align*}
let $\gamma_{\mathrm{dist}}:=\mu_{\mathrm{in}}-\mu_{\mathrm{out}}$, and suppose all variances are bounded by $\sigma^2$.  If
\begin{equation}
        \gamma_{\mathrm{dist}} s
        \ge
        C_{\mathrm{SDP}}\max\left\{
        \sqrt{\sigma^2 n},
        \sqrt{\sigma^2 s\log(n)},
        \log(n)
        \right\},
\label{eq:sdp-condition}
\end{equation}
then, with probability at least $1-n^{-c}$, the matrix $X^0:=\frac1sY^{\sigma_0}$ is the unique optimizer of
\begin{equation}
\max_X \langle W,X\rangle
\quad\text{subject to}\quad
X\succeq 0,
\quad X\one\le \one,
\quad \Tr(X)=K,
\quad X\ge 0.
\label{eq:weighted-sdp}
\end{equation}
\end{lemma}

For $Y\in\cC$, we have $\sum_{i<j}Y_{ij}=\frac{ns-n}{2}$, so the penalty $\lambda\sum_{i<j}Y_{ij}$ is constant over $\cC$. Let $Y^0:=Y^{\sigma_0}$. By Lemma \ref{lem:truth-agreement}, on $\cE_D$, we have
\[
        \Psiin_{A,D}(Y^0)
        =
        \sum_{ij\in E(A)}Y^0_{ij}-\lambda\sum_{i<j}Y^0_{ij}.
\]
By inequality~\eqref{EqnHAD} in the proof of Lemma \ref{lem:truth-agreement}, we have
\[
        \Psiin_{A,D}(Y)
        \le
        \sum_{ij\in E(A)}Y_{ij}-\lambda\sum_{i<j}Y_{ij},
\]
for all $Y\in\cC$.
Since the penalty is constant over $\cC$, it suffices to show that
\begin{equation}
        \sum_{i<j}A_{ij}Y_{ij}+\frac\theta n\|Y-Y^0\|_1
        \le
        \sum_{i<j}A_{ij}Y^0_{ij},
        \qquad\forall Y\in\cC.
\label{eq:linear-eroded-suffices}
\end{equation}
Using equation~\eqref{eq:l1-affine} with $Z=Y^0$, define $\delta_n:=\frac{2\theta}{n}$ and the shifted weights
\begin{equation}
        B_{ij}:=
        \begin{cases}
        A_{ij}-\delta_n,& \sigma_0(i)=\sigma_0(j),\\
        A_{ij}+\delta_n,& \sigma_0(i)\ne\sigma_0(j).
        \end{cases}
\label{eq:B-shift}
\end{equation}
Then inequality~\eqref{eq:linear-eroded-suffices} is equivalent to saying that $Y^0$ maximizes $Y\mapsto \sum_{i<j}B_{ij}Y_{ij}$ over $\cC$.  More explicitly,
\[
\sum_{i<j}A_{ij}Y_{ij}+\frac\theta n\|Y-Y^0\|_1
= C_0' + \sum_{i<j}B_{ij}Y_{ij},
\]
where $C_0'$ is independent of $Y$, and equality at $Y=Y^0$ is exact.

It remains to prove that $Y^0$ is the unique maximizer of the shifted objective with high probability.  Since $B_{ij}$ may be slightly negative or slightly larger than one, define
\begin{equation}
        W_{ij}:=\frac{B_{ij}+\delta_n}{1+2\delta_n}.
\label{eq:W-transform}
\end{equation}
Adding $\delta_n$ to every off-diagonal weight and multiplying by the positive constant $(1+2\delta_n)^{-1}$ does not change the maximizer over $\cC$, because $\sum_{i<j}Y_{ij}$ is constant over $\cC$.  Moreover, $0\le W_{ij}\le1$.  Conditional on $\sigma_0$, the weights are independent and
\[
\mu_{\mathrm{in}}
=
\frac{a/n}{1+2\delta_n},
\qquad
\mu_{\mathrm{out}}
=
\frac{b/n+2\delta_n}{1+2\delta_n}.
\]
Thus,
\[
        \gamma_{\mathrm{dist}}:=\mu_{\mathrm{in}}-\mu_{\mathrm{out}}
        =
        \frac{a-b-4\theta}{n(1+2\delta_n)}
        \ge
        c \cdot \frac an,
\]
since $a-b-4\theta\ge c_\rho a$ and $a=o(n)$. Also, all variances are bounded by $Ca/n$, implying that $\gamma_{\mathrm{dist}} s\ge c\frac aK$. On the other hand, we have
\[
        \sqrt{\sigma^2 n}\le C\sqrt a,
        \qquad
        \sqrt{\sigma^2 s\log(n)}\le C\sqrt{\frac{a\log(n)}{K}}.
\]
Write $S:=a/K$.  The high-signal condition gives $S\ge A_0\log(nK)\ge A_0\log(n)$. Assumption \ref{ass:mild-K} implies $K\lesssim \log(n)/\log(\log(n))$, so for $A_0$ large and all sufficiently large $n$, we have
\[
        S\ge C\sqrt{KS}=C\sqrt a,
        \qquad
        S\ge C\sqrt{S\log(n)},
        \qquad
        S\ge C\log(n).
\]
Hence, condition \eqref{eq:sdp-condition} holds for the shifted weights $W$. Lemma \ref{lemma:weighted-sdp} implies that $X^0=Y^0/s$ is the unique optimizer of \eqref{eq:weighted-sdp}, on an event $\Omega_{PA}$, with $\mathbb{P}(\Omega_{PA}) \ge 1 - n^{-c}$. Since every $Y\in\cC$ gives $X=Y/s$ feasible for \eqref{eq:weighted-sdp}, $Y^0$ is the unique maximizer of the shifted objective over $\cC$. This proves inequality~\eqref{eq:linear-eroded-suffices}, and hence also the certification condition~\eqref{EqnCertify}, on $\cE_D \cap \Omega_{PA}$. The probability bound follows from $\Pp_{\sigma_0}(\cE_D^c)\le(nK)^{-c_{deg} A_0}$, using Lemma \ref{lem:true-within-conc}.

%%%%%

\subsection{Proof of Theorem~\ref{thm:runtime}}
\label{AppThmRuntime}

The first statement holds by construction.  If the certificate branch is used, exactness follows from Lemma \ref{lem:sampler-exact}; if not, exactness follows from brute-force enumeration.  Since the output law is $\pi_A$ for every input graph, privacy follows from Lemma \ref{lem:global-sensitivity}.

It remains to prove the runtime statement. By Lemma~\ref{lemma:weighted-sdp} and the same calculation as in Lemma~\ref{lem:eroded-cert-highprob}, but without the shift, the candidate SDP
recovers \(X^0=Y^0/s\) with probability at least \(1-n^{-c}\). Indeed, for the unshifted weights \(W_{ij}=A_{ij}\), the mean gap is
\((a-b)/n\asymp a/n\) and the variance is at most \(a/n\), so condition \eqref{eq:sdp-condition} follows from Assumptions~\ref{ass:mild-K} and~\ref{ass:high-signal}, with \(A_0\) large. Let
\[
        \Omega_{\mathrm{cand}}
        :=
        \{s\widehat X=Y^0\}
\]
denote the candidate recovery event. The preceding argument shows that $\mathbb{P}(\Omega_{\mathrm{cand}}) \ge 1 - n^{-c}$. By the proof of Lemma~\ref{lem:eroded-cert-highprob}, on
\(\cE_D\cap\Omega_{PA}\), where
\[
        \Pp_{\sigma_0}(\cE_D\cap\Omega_{PA})
        \ge
        1-n^{-c}-(nK)^{-cA_0},
\]
the eroded certificate succeeds at \(Y^0\). Define $\Omega_{\mathrm{poly}}:=\Omega_{\mathrm{cand}}\cap\cE_D\cap\Omega_{PA}$. By a union bound, after adjusting constants, we have
\[
        \Pp_{\sigma_0}(\Omega_{\mathrm{poly}})
        \ge
        1-n^{-c}-(nK)^{-cA_0}.
\]
On \(\Omega_{\mathrm{poly}}\), the candidate satisfies
\(Y^{\sigma^\star}=Y^0\), and the eroded certificate succeeds at \(Y^0\). Hence, we have
\[
        V_\theta(Y^{\sigma^\star})
        =
        V_\theta(Y^0)
        =
        \Psiin_{A,D}(Y^0)
        =
        \Psiin_{A,D}(Y^{\sigma^\star}).
\]

Moreover, $\kappa=\eta\frac{2\theta}{K}=\frac{\varepsilon}{2D}\frac{2\theta}{K} = \frac{\varepsilon\theta}{DK}$. Thus, choosing \(L\) sufficiently large, the assumption
\(\varepsilon\ge L\log(nK)\) implies $\kappa\ge \log(4n(K-1))$.
%So, parameter condition $\kappa\ge \log(4n(K-1))$ for Algorithm~\ref{AlgOverview} to run is satisfied.
On \(\Omega_{\mathrm{poly}}\), condition~\eqref{EqnCertify} holds, so the algorithm
enters the certified branch.

Lemma \ref{lem:sampler-exact} gives an expected number of proposal trials at most $e^{1/4}$, and each proposal can be processed in polynomial time. Indeed, each trial requires sampling $n$ labels, and checking membership in $\Sigma$ takes polynomial time. Next, checking canonicity can be done by solving one $K\times K$ assignment problem, and this can be executed in polynomial time by Lemma \ref{lem:hungarian-canon}. Evaluating $\Tin_{A,D}$ is polynomial, by Remark~\ref{lem:score-eval}. The candidate SDP and eroded certificate SDP are polynomial-size convex programs under the exact-oracle convention, by Lemma \ref{lem:eroded-sdp}. Moreover, sampling a uniform permutation in Algorithm \ref{alg:certified-sampler} is polynomial-time computable, by Remark \ref{remark:sampling_poly_unif_perm}.

Let \(\mathsf U\) denote the internal randomness of the mechanism. For a
deterministic input graph \(H\), let \(\mathrm{Runtime}(H,\mathsf U)\) denote
the runtime of the full mechanism on input \(H\). By the preceding argument, for every \(\omega\in\Omega_{\mathrm{poly}}\), the
full mechanism run on \(A(\omega)\) enters the certified rejection sampler
branch. Lemma \ref{lem:sampler-exact} gives an expected number of proposal trials at most
\(e^{1/4}\), and each proposal is processed in polynomial time. The candidate recovery
program is an SDP, and by Lemma \ref{lem:eroded-sdp}, the eroded certificate check is also a polynomial-size SDP. Hence, since the candidate SDP and the eroded certificate SDP are polynomial-size convex
programs, there exists a polynomial \(p(n)\) such that
\[
\mathbb E_{\mathsf U}
\left[
\mathrm{Runtime}(A(\omega),\mathsf U)
\right]
\le p(n),
\qquad
\forall \omega\in\Omega_{\mathrm{poly}}.
\]
Equivalently, we have
\[
\Omega_{\mathrm{poly}}
\subseteq
\left\{
\omega:
\mathbb E_{\mathsf U}
\left[
\mathrm{Runtime}(A(\omega),\mathsf U)
\right]
\le p(n)
\right\}.
\]
Since
\[
\mathbb P_{\sigma_0}(\Omega_{\mathrm{poly}})
\ge
1-n^{-c} - (nK)^{-cA_0},
\]
we obtain
\[
\mathbb P_{\sigma_0}
\left(
\mathbb E_{\mathsf U}
\left[
\mathrm{Runtime}(A,\mathsf U)
\right]
\le p(n)
\right)
\ge
1-n^{-c} - (nK)^{-cA_0}.
\]
This proves the high-probability polynomial expected runtime claim and
completes the proof.

%%%%%

\subsection{Proof of Theorem~\ref{thm:risk}}
\label{AppUtilityProof}

\subsubsection{Supporting lemmas}

The first lemma records the entropy estimate needed for the peeling argument. It is obtained from the proof of Lemma A.5 in \cite{klopp2026node}, specialized to our equal-sized community setting.

\begin{lemma}
\label{lemma:raw-kz}
Under Assumptions \ref{ass:mild-K} and \ref{ass:high-signal}, there are constants $C_0,C_1,c_0>0$ such that, with probability at least $1-e^{-c_0nI/K}$, we have
\begin{equation}
\sup_{u \ge 0} \left\{\log\left(\left|
\left\{\sigma\in\Sigma:
T_A(\sigma)\ge T_A(\sigma_0)-u
\right\}
\right|\right) - Bu\right\}
\le
C_1\log(nK),
\label{eq:kz-entropy}
\end{equation}
where $B := C_0\frac{K\log(nK)}{nI}$.
\end{lemma}

\begin{proof}
The proof of this result follows from the proof of Lemma A.5 in \cite{klopp2026node}. Although that lemma is stated for the near-maximum set
\[
S_u(A)=\{\sigma:T_A(\sigma)\ge \max_{\tau\in\Sigma}T_A(\tau)-u\},
\]
their proof introduces
\[
\widetilde S_u(A)=\{\sigma:T_A(\sigma)\ge T_A(\sigma_0)-u\}
\]
and proves the displayed entropy bound for \(\widetilde S_u(A)\), using
\(S_u(A)\subseteq \widetilde S_u(A)\).
\end{proof}

The next result is the analog of the ``peeling result," Lemma A.6, in \cite{klopp2026node}.

\begin{lemma}
\label{lem:peeling}
Suppose Assumptions \ref{ass:mild-K} and  \ref{ass:high-signal} hold. Let $\widehat\sigma\sim\pi_A$. Let $B = C_0\frac{K\log(nK)}{nI}$, and suppose $\gamma_0:=\eta-2B >0$.  For $0<\alpha<1/2$, set
\begin{equation}
        u_\star(\alpha)
        :=
        \frac{C_2\log(nK)+\log(4/\alpha)}{\gamma_0},
\label{eq:u-star}
\end{equation}
where $C_2$ is a sufficiently large absolute constant.  Then
\begin{equation*}
\Pp\left(
\cE_D\cap
\left\{
\Tin_{A,D}(\widehat\sigma)
<
\max_{\tau\in\Sigma}\Tin_{A,D}(\tau)-u_\star(\alpha)
\right\}
\right)
\le
\alpha+e^{-c_0nI/K}.
\end{equation*}
\end{lemma}

\begin{proof}
Condition on a graph $A$ for which $\cE_D$ and the event \eqref{eq:kz-entropy} both hold.  For $u\ge0$, define
\[
        \cS^{\mathrm{tilde}}_u(A)
        :=
        \left\{\sigma\in\Sigma:
        \Tin_{A,D}(\sigma)
        \ge
        \max_\tau\Tin_{A,D}(\tau)-u
        \right\}.
\]
By Lemma \ref{lem:truth-agreement}, on $\cE_D$, we have
\[
        \cS^{\mathrm{tilde}}_u(A)
        \subseteq
        \left\{\sigma\in\Sigma:
        T_A(\sigma)\ge T_A(\sigma_0)-u
        \right\}.
\]
By Lemma \ref{lemma:raw-kz}, we have $\log( |\cS^{\mathrm{tilde}}_u(A)|)\le Bu+C_1\log(nK)$. Let $M:=\max_\tau\Tin_{A,D}(\tau)$. A standard exponential-mechanism peeling bound implies, for any $s>0$, that
\begin{align}
\Pp(M-\Tin_{A,D}(\widehat\sigma)>s\mid A) &\le \sum_{\ell\ge1}
\left|\cS^{\mathrm{tilde}}_{(\ell+1)s}(A)\right|
\exp\{-\eta\ell s\} \le
(nK)^{C_1}
\sum_{\ell\ge1}
\exp\{B(\ell+1)s-\eta\ell s\} \notag\\
&= (nK)^{C_1}\frac{\exp\{-(\eta-2B)s\}}{1-\exp\{-(\eta-B)s\}}.
\label{eq:M-TtildeS}
\end{align}
Indeed, for \(\ell\ge1\), set $\mathcal A_\ell:=\left\{\sigma\in\Sigma:\ell s<M-\Tin_{A,D}(\sigma)\le(\ell+1)s\right\}$. Then $$\{M-\Tin_{A,D}(\widehat\sigma)>s\}\subseteq\bigcup_{\ell\ge1}\{\widehat\sigma\in\mathcal A_\ell\}.$$ For \(\sigma\in\mathcal A_\ell\), we have
\[
        \pi_A(\sigma)
        =
        \frac{\exp\{\eta\Tin_{A,D}(\sigma)\}}
        {\sum_{\tau\in\Sigma}\exp\{\eta\Tin_{A,D}(\tau)\}}
        \le
        \frac{\exp\{\eta(M-\ell s)\}}{\exp\{\eta M\}}
        =
        e^{-\eta\ell s}.
\]
We also have $\mathcal A_\ell\subseteq\cS^{\mathrm{tilde}}_{(\ell+1)s}(A)$. This proves the first inequality in \eqref{eq:M-TtildeS}. The second inequality of \eqref{eq:M-TtildeS} follows from the entropy bound on
\(\cS^{\mathrm{tilde}}_u(A)\) with \(u=(\ell+1)s\). Taking \(s=u_\star(\alpha)\) and increasing \(C_2\) if necessary, we have
\[
(nK)^{C_1}\frac{\exp\{-(\eta-2B)s\}}{1-\exp\{-(\eta-B)s\}} \le \alpha,
\]
since \(\eta-B\ge \eta-2B=\gamma_0\), and $\alpha < 1/2$. Averaging over $A$ and adding the probability $e^{-c_0nI/K}$ that inequality~\eqref{eq:kz-entropy} fails proves the claim.
\end{proof}

The next result gives a bound on $\mathbb E r(\sigma_0,\widehat\sigma)$, provided $\widehat\sigma$ satisfies $T_A(\widehat\sigma) \ge T_A(\sigma_0)-u$. This is similar in spirit to Lemma A.1 of~\cite{klopp2026node}, where they prove their result under the condition $T_A(\widehat\sigma) \ge \max_{\tau \in \Sigma}T_A(\tau)-u$. 

\begin{lemma}
\label{lem:true-relative-raw-slack} 
Suppose Assumptions \ref{ass:mild-K} and \ref{ass:high-signal} hold, with $C_{\mathrm{mg}}$ sufficiently small, and $A_0$ sufficiently large. Let 
\[ 
t^\star := \frac12 \log\left( \frac{a\left(1-\frac{b}{n}\right)}{b\left(1-\frac{a}{n}\right)} \right). 
\] 
Let $\widehat\sigma\in\Sigma$ be a (randomized) estimator. Suppose that for some deterministic \(u\ge 0\), we have $T_A(\widehat\sigma) \ge T_A(\sigma_0)-u$, almost surely. Then
%uniformly over the exact-balanced parameter class described above, 
\[ 
\mathbb E r(\sigma_0,\widehat\sigma) \le \exp\left\{ -(1+o(1))\frac{nI}{K} + t^\star u \right\}.
\] 
%The $o(1)$ term is the same uniform $o(1)$ term appearing in the non-private Klopp--Zadik/Zhang--Zhou layer-counting analysis under \eqref{eq:constant-snr}, and Assumptions \ref{ass:mild-K}, \ref{ass:high-signal}. 
\end{lemma} 

\begin{proof}
All probabilities and expectations are under the SBM with ground truth
$\sigma_0$, together with any auxiliary randomness used by $\widehat \sigma$. For $m=1,\dots,n$, define
\[
\mathcal E_m(u)
:=
\left\{
\exists \sigma\in\Sigma:
\dorb([\sigma], [\sigma_0])=m
\text{ and }
T_A(\sigma)\ge T_A(\sigma_0)-u
\right\}.
\]
Since $T_A(\widehat\sigma)\ge T_A(\sigma_0)-u$ almost surely and
$\widehat\sigma\in\Sigma$, we have
\[
\{\dorb([\widehat\sigma], [\sigma_0])=m\}\subseteq \mathcal E_m(u),
\qquad m=1,\dots,n.
\]
Therefore, we have
\begin{equation}
\mathbb E r(\sigma_0,\widehat\sigma)
\le
\frac1n\sum_{m=1}^n m\,\mathbb P(\mathcal E_m(u)).
\label{eq:true-slack-risk-layer-sum}
\end{equation}
Both $T_A(\sigma)$ and $\dorb([\sigma], [\sigma_0])$ are constant on $[\sigma]$.
Let
\[
\mathcal G_m
:=
\{[\sigma]:\sigma\in\Sigma,\ \dorb([\sigma], [\sigma_0])=m\}.
\]
For each $\Gamma\in\mathcal G_m$, choose one representative
$\sigma_\Gamma\in\Gamma$. Then
\[
\mathcal E_m(u)
=
\bigcup_{\Gamma\in\mathcal G_m}
\{T_A(\sigma_\Gamma)\ge T_A(\sigma_0)-u\}.
\]
Fix $\sigma\in\Sigma$, and note that $\sum_{i<j}\mathbf 1\{\sigma(i)=\sigma(j)\}=K\binom{n/K}{2}$ is independent of $\sigma$. Hence the $\lambda$-penalty cancels in
$T_A(\sigma)-T_A(\sigma_0)$. Define
\[
S(\sigma;\sigma_0)
:=
\{i<j:\sigma_0(i)=\sigma_0(j),\ \sigma(i)\neq\sigma(j)\},
\]
\[
M(\sigma;\sigma_0)
:=
\{i<j:\sigma_0(i)\neq\sigma_0(j),\ \sigma(i)=\sigma(j)\},
\]
and
\[
\alpha(\sigma;\sigma_0):=|S(\sigma;\sigma_0)|,
\qquad
\gamma(\sigma;\sigma_0):=|M(\sigma;\sigma_0)|.
\]
Exact balance gives $\alpha(\sigma;\sigma_0)=\gamma(\sigma;\sigma_0)$. Indeed, the total number of within-community pairs is the same under
$\sigma$ and $\sigma_0$, so the number of true within-pairs lost must equal
the number of true between-pairs gained. Under $\sigma_0$, we have
\[
T_A(\sigma)-T_A(\sigma_0)
\stackrel{d}{=}
\sum_{\ell=1}^{\gamma(\sigma;\sigma_0)} X_\ell
-
\sum_{\ell=1}^{\alpha(\sigma;\sigma_0)} Y_\ell,
\]
where $X_\ell\stackrel{\mathrm{iid}}{\sim}\mathrm{Ber}(b/n)$ and $Y_\ell\stackrel{\mathrm{iid}}{\sim}\mathrm{Ber}(a/n)$ are independent. By Markov's inequality, we have
\[
\mathbb P_{\sigma_0}
\left(T_A(\sigma)-T_A(\sigma_0)\ge -u\right)
\le
e^{t^\star u}
\mathbb E_{\sigma_0}
\exp\left\{
t^\star\bigl(T_A(\sigma)-T_A(\sigma_0)\bigr)
\right\}.
\]
Using $\alpha=\gamma$ and the definition of the order-$1/2$ R\'{e}nyi divergence
$I$, we have
\[
\mathbb E_{\sigma_0}
\exp\left\{
t^\star\bigl(T_A(\sigma)-T_A(\sigma_0)\bigr)
\right\}
=
\left[
\mathbb E e^{t^\star X_1}
\mathbb E e^{-t^\star Y_1}
\right]^{\alpha(\sigma;\sigma_0)}
=
e^{-I\alpha(\sigma;\sigma_0)}.
\]
Thus, with $Q(\sigma,\sigma_0):=e^{-I\alpha(\sigma;\sigma_0)}$, we have
\[
\mathbb P_{\sigma_0}
\left(T_A(\sigma)-T_A(\sigma_0)\ge -u\right)
\le
e^{t^\star u}Q(\sigma,\sigma_0).
\]
Taking a union bound over orbits gives
\begin{equation*}
\mathbb P(\mathcal E_m(u))
\le
e^{t^\star u}
\sum_{\Gamma\in\mathcal G_m}
Q(\sigma_\Gamma,\sigma_0).
\end{equation*}
Substituting this into inequality~\eqref{eq:true-slack-risk-layer-sum}, we obtain
\begin{equation}
\mathbb E r(\sigma_0,\widehat\sigma)
\le
e^{t^\star u}
\frac1n
\sum_{m=1}^n
m
\sum_{\Gamma\in\mathcal G_m}
Q(\sigma_\Gamma,\sigma_0).
\label{eq:true-slack-orbit-reduction}
\end{equation}

It remains to prove that
\begin{equation}
\frac1n
\sum_{m=1}^n
m
\sum_{\Gamma\in\mathcal G_m}
Q(\sigma_\Gamma,\sigma_0)
\le
\exp\left\{-(1+o(1))\frac{nI}{K}\right\}.
\label{eq:needed-layer-sum}
\end{equation}
Write $S_I :=  \mathrm{Signal}=\frac{nI}{K}$. Under the sparse constant-SNR regime (Assumption \ref{ass:high-signal}), we have $I=o(1)$ and $nI\asymp a$. By Assumption~\ref{ass:high-signal}, after increasing its absolute constant if necessary, we may assume that $S_I\ge A_1\log(nK)$ for an arbitrarily large fixed constant $A_1$. Also, Assumption~\ref{ass:mild-K}
implies $K=o(n)$ and $\log(K)=o(S_I)$.

First consider $K=2$. Since distance is computed modulo a global relabeling, we have
$\mathcal G_m=\varnothing$ for $m>n/2$. For each
$\Gamma\in\mathcal G_m$, choose $\sigma_\Gamma$ so that
$h(\sigma_\Gamma,\sigma_0)=m$. The two-community split/merge identity gives
\[
\alpha(\sigma_\Gamma;\sigma_0)+\gamma(\sigma_\Gamma;\sigma_0)
=
m(n-m).
\]
Since exact balance gives $\alpha=\gamma$, we have $\alpha(\sigma_\Gamma;\sigma_0)=\frac{m(n-m)}2$. Thus, we have
\[
Q(\sigma_\Gamma,\sigma_0)
=
\exp\left\{-\frac I2m(n-m)\right\}.
\]
Moreover, $|\mathcal G_m|\le \binom nm$, implying that
\[
\frac1n
\sum_{m=1}^n
m
\sum_{\Gamma\in\mathcal G_m}
Q(\sigma_\Gamma,\sigma_0)
\le
\frac1n
\sum_{m=1}^{\lfloor n/2\rfloor}
m\binom nm
\exp\left\{-\frac I2m(n-m)\right\},
\]
where $S_I=nI/2$. The $m=1$ term is at most
\[
\exp\left\{-\frac I2(n-1)\right\}
=
\exp\{-(1+o(1))S_I\}.
\]
For $2\le m\le n/4$, we have $\frac I2m(n-m)\ge \frac34 S_I$, while $\frac mn\binom nm\le \left(\frac{en}{m}\right)^m$. Taking $A_1$ large enough gives
\[
\left(\frac{en}{m}\right)^m e^{-3mS_I/4}
\le
e^{-5mS_I/8},
\]
uniformly over $2\le m\le n/4$. Therefore, the contribution of
$2\le m\le n/4$ is $o(e^{-S_I})$. For $n/4<m\le n/2$, using $m/n \le 1$, the crude bounds $\sum_{n/4 < m \le n/2}\binom nm\le 2^n$ and $m(n-m)\ge n^2/8$ give a total contribution bounded
by
\[
\exp\{n\log(2)-nS_I/8\}=o(e^{-S_I}).
\]
Thus, inequality~\eqref{eq:needed-layer-sum} holds for $K=2$.

Now assume $K\ge 3$. By Lemma B.5 in \cite{klopp2026node}, specialized to exact balance, we have for every
$\sigma\in\Sigma$ with $\dorb([\sigma], [\sigma_0])=m$, that
\[
\alpha(\sigma;\sigma_0)
\ge
\begin{cases}
\dfrac{nm}{K}-m^2, & m\le \dfrac{n}{2K}, \\[1.2ex]
c_\star \dfrac{nm}{K}, & m> \dfrac{n}{2K},
\end{cases}
\]
where $c_\star>0$ is an absolute constant. This follows from the usual bound
on $\alpha\wedge\gamma$, because exact balance gives $\alpha=\gamma$. We also use the elementary bound $|\mathcal G_m|\le \left(\frac{enK}{m}\right)^m$.
Indeed, an orbit at distance $m$ has a representative obtained by choosing
the $m$ moved vertices and assigning each of them one of at most $K$ labels. Split the sum into three ranges:
\begin{enumerate}
\item If $1\le m\le n/(4K)$, then $\alpha(\sigma_\Gamma;\sigma_0)\ge\frac{3nm}{4K}$, and hence,
\[
\frac mn|\mathcal G_m|e^{-I\alpha(\sigma_\Gamma;\sigma_0)}
\le
\frac mn
\left(\frac{enK}{m}\right)^m
e^{-3mS_I/4}.
\]
Since $S_I\ge A_1\log(nK)$ with $A_1$ sufficiently large, we have
\[
\left(\frac{enK}{m}\right)^m e^{-3mS_I/4}
\le
e^{-5mS_I/8}.
\]
The $m=1$ term is at most $K e^{-3S_I/4}$. This is not sharp enough by itself, so for $m=1$, we use the sharper bound
$\alpha\ge n/K-1$, which gives
\[
\frac1n|\mathcal G_1|e^{-I\alpha}
\le
K\exp\{-I(n/K-1)\}
=
\exp\{-(1+o(1))S_I+\log(K)\}
=
\exp\{-(1+o(1))S_I\},
\]
because $\log(K)=o(S_I)$. The remaining terms $m\ge2$ in this range are
$o(e^{-S_I})$.

\item If $n/(4K)<m\le n/(2K)$, then $\alpha(\sigma_\Gamma;\sigma_0)
\ge
\frac{nm}{2K}$, so
\[
\frac mn|\mathcal G_m|e^{-I\alpha}
\le
\left(\frac{enK}{m}\right)^m e^{-mS_I/2}.
\]
Again, for $A_1$ sufficiently large, we have
\[
\left(\frac{enK}{m}\right)^m e^{-mS_I/2}
\le e^{-mS_I/4}.
\]
Since $m>n/(4K)$ and $K=o(n)$, the total contribution of this range is
$o(e^{-S_I})$.

\item Finally, if $m>n/(2K)$, then $\alpha(\sigma_\Gamma;\sigma_0)\ge c_\star \frac{nm}{K}$ and $\frac{enK}{m}\le 2eK^2$, Therefore, we have
\[
\frac mn|\mathcal G_m|e^{-I\alpha}
\le
\exp\{m\log(2eK^2)-c_\star mS_I\}.
\]
Since $\log(K)=o(S_I)$, for all sufficiently large $n$, we have $\log(2eK^2)\le \frac{c_\star}{2}S_I$. Thus, each term in this range is at most $e^{-c_\star mS_I/2}$. Since $m>n/(2K)$ and $K=o(n)$, the whole contribution is $o(e^{-S_I})$.
\end{enumerate}

Combining the three ranges gives inequality~\eqref{eq:needed-layer-sum} for $K\ge3$, implying the desired bound.
%Together with the already proved $K=2$ case, inequality~\eqref{eq:needed-layer-sum}
%holds for every $K\ge2$.

Finally, substituting inequality~\eqref{eq:needed-layer-sum} into
inequality~\eqref{eq:true-slack-orbit-reduction} yields
\[
\mathbb E r(\sigma_0,\widehat\sigma)
\le
\exp\left\{
-(1+o(1))\frac{nI}{K}
+
t^\star u
\right\}.
\]
Taking the supremum over $\sigma_0\in\Sigma$ proves the claim.
\end{proof}

%%%%%

\subsubsection{Proof of theorem}

By Theorem~\ref{thm:runtime}, the output law of Algorithm~\ref{AlgOverview} is exactly $\pi_A$, so Lemma \ref{lem:peeling} applies.  Since $D=C_{deg}a/K$ and $nI\asymp a$, we have
\[
        B=C_0\frac{K\log(nK)}{nI}
        \asymp
        \frac{K\log(nK)}{a}.
\]
Also, $\eta=\frac{\varepsilon}{2D}\asymp\frac{\varepsilon K}{a}$. Thus, for $L$ sufficiently large, $\varepsilon\ge L\log(nK)$ implies
$\eta\ge 4B$, so $\gamma_0=\eta-2B\asymp\eta\asymp\frac{\varepsilon}{D}$. Choose $ \alpha=(nK)^{-1}e^{-c\varepsilon}$, with $c>0$ small.  Then equation~\eqref{eq:u-star} gives
\[
        u_\star(\alpha)
        \lesssim
        D\frac{\log(nK)+c\varepsilon}{\varepsilon}
        \asymp
        \frac aK\left(\frac{\log(nK)}{\varepsilon}+c\right).
\]
Since $nI/K\asymp a/K$ and $t^\star=\Theta(1)$, choosing $L$ large and then $c$ small gives
\begin{equation}
        t^\star u_\star(\alpha)
        \le
        c'\frac{nI}{K},
\label{eq:slack-small}
\end{equation}
for a sufficiently small constant $c'>0$.

Let
\[
        U
        :=
        \cE_D\cap
        \left\{
        \Tin_{A,D}(\widehat\sigma)
        \ge
        \max_{\tau}\Tin_{A,D}(\tau)-u_\star(\alpha)
        \right\}.
\]
By Lemmas \ref{lem:true-within-conc} and \ref{lem:peeling}, we have $\Pp(U^c)\le\alpha+e^{-c_0nI/K}+(nK)^{-c_{deg}A_0}$. On $U$, Lemma \ref{lem:truth-agreement} gives $T_A(\widehat\sigma)\ge T_A(\sigma_0)-u_\star(\alpha)$. Now define a hybrid estimator
\[
        \widetilde\sigma(A)
        :=
        \begin{cases}
        \widehat\sigma(A),& A\in U,\\
        \sigma^{\mathrm{raw}}(A),& A\notin U,
        \end{cases}
\]
where $\sigma^{\mathrm{raw}}(A)\in\arg\max_{\sigma\in\Sigma}T_A(\sigma)$. Then $T_A(\widetilde\sigma)\ge T_A(\sigma_0)-u_\star(\alpha)$ always. Lemma \ref{lem:true-relative-raw-slack}, and \eqref{eq:slack-small} yield
\[
        \Ee r(\sigma_0,\widetilde\sigma)
        \le
        \exp\left\{-c_1\frac{nI}{K}\right\},
\]
for some $c_1>0$. Since $r\le1$, we have $\Ee r(\sigma_0,\widehat\sigma) \le \Ee r(\sigma_0,\widetilde\sigma)+\Pp(U^c)$. Finally, we have
\[
        \alpha=(nK)^{-1}e^{-c\varepsilon},
        \qquad
        e^{-c_0nI/K}
        \le
        \exp\left\{-c\frac{nI}{K}\right\},
\]
and the within-degree error is bounded by Lemma \ref{lem:true-within-conc}.  Absorbing constants proves the bound~\eqref{eq:risk-final}.

%%%%%

\section{Proofs for Section~\ref{sec:truncated-sampler}}

\subsection{Proof of Lemma~\ref{lem:truncated-uniform-comparison}}
\label{AppTrunComp}

Given an input graph $A$, Algorithm~\ref{AlgOverview} produces a candidate $\sigma^\star(A)\in\Sigma$. Define
\[
    {\mathsf C}(A)
    :=
    \mathbf 1\left\{
        V_\theta(Y^{\sigma^\star(A)})
        =
        \widetilde\Psi_{A,D}(Y^{\sigma^\star(A)})
        \ \text{and}\
        \kappa\ge \log(4n(K-1))
    \right\}.
\]
When ${\mathsf C}(A)=1$, the modified algorithm runs the truncated
certified rejection sampler (Algorithm~\ref{alg:truncated-certified-sampler}) with candidate $\sigma^\star(A)$. When
${\mathsf C}(A)=0$, the modified algorithm samples exactly from $\pi_A$ by
brute-force enumeration. Thus, we have
\[
    Q_A^{(M)}
    =
    \begin{cases}
        \text{law of Algorithm~\ref{alg:truncated-certified-sampler}
        with candidate $\sigma^\star(A)$}, & {\mathsf C}(A)=1,\\[3pt]
        \pi_A, & {\mathsf C}(A)=0.
    \end{cases}
\]

We begin with the following result:

\begin{lemma}
\label{lem:truncated-certified-decomposition}
%For every input graph $A$, the following statements hold:
%\begin{enumerate}
%    \item If ${\mathsf C}(A)=0$, then $Q_A^{(M)}=\pi_A$.
Suppose ${\mathsf C}(A)=1$, and
%$\sigma^*(A)$ denote the candidate used in the algorithm.
%and define $r^A_{\rm acc} := (1-p^A_{\rm acc})^M$.
let $R_A$ denote the uniform distribution on the orbit
$[\sigma^\star(A)]$. Then $R_A(\sigma)\le e^{1/4}\pi_A(\sigma)$, for all $\sigma\in\Sigma$, and
\begin{equation}
\label{EqnQAM2}
    Q_A^{(M)}
    =
    (1-r^A_{\rm acc})\pi_A+r^A_{\rm acc} R_A,
\end{equation}
where
$r^A_{\rm acc}\le \overline r_M$.
%\end{enumerate}
\end{lemma}

\begin{proof}
%If ${\mathsf C}(A)=0$, the modified Algorithm \ref{alg:full} samples exactly from $\pi_A$ by
%brute-force enumeration, so $Q_A^{(M)}=\pi_A$.
Suppose ${\mathsf C}(A)=1$, and abbreviate
$\sigma^\star=\sigma^\star(A)$. Let $G=\bigl(1+(K-1)e^{-\kappa}\bigr)^n$. The proposal law in one trial is
\[
    \nu_A(\sigma)
    =
    G^{-1}e^{-\kappa h(\sigma,\sigma^\star)},
    \qquad \forall \sigma\in[K]^n,
\]
and the probability of acceptance is
\[
    a_A(\sigma)
    :=
    \begin{cases}
    \exp\left\{
        \eta\bigl[
            \widetilde T_{A,D}(\sigma)
            -
            \widetilde T_{A,D}(\sigma^\star)
        \bigr]
        +
        \kappa h(\sigma,\sigma^\star)
    \right\},
    & \sigma\in \Sigma\cap R_{\sigma^\star},\\[3pt]
    0, & \text{otherwise}.
    \end{cases}
\]
Since ${\mathsf C}(A)=1$, Lemma \ref{lem:certificate-margin} implies $0\le a_A(\sigma)\le 1$. Hence, the one-trial acceptance probability is $p^A_{\rm acc} = \sum_{\sigma\in[K]^n}\nu_A(\sigma)a_A(\sigma)$. For $\sigma\in\Sigma\cap R_{\sigma^\star}$, we have
\[
    \nu_A(\sigma)a_A(\sigma)
    =
    G^{-1}
    \exp\left\{
        \eta\bigl[
            \widetilde T_{A,D}(\sigma)
            -
            \widetilde T_{A,D}(\sigma^\star)
        \bigr]
    \right\},
\]
implying that $p^A_{\rm acc} = G^{-1}Z_A$, where
\[
    Z_A
    :=
    \sum_{\sigma\in\Sigma\cap R_{\sigma^\star}}
    \exp\left\{
        \eta\bigl[
            \widetilde T_{A,D}(\sigma)
            -
            \widetilde T_{A,D}(\sigma^\star)
        \bigr]
    \right\}.
\]
Since $\sigma^\star\in\Sigma\cap R_{\sigma^\star}$, we clearly have $Z_A\ge 1$.

Let $B_M$ denote the event that at least one of the first $M$ proposal trials is
accepted. Then
\[
    \mathbb P(B_M^c)=(1-p^A_{\rm acc})^M :=r^A_{\rm acc}.
\]
We next compute the law conditional on $B_M$. Let $\widehat\sigma_R$ denote the
accepted canonical representative before the final random relabeling. For
$\sigma\in\Sigma\cap R_{\sigma^\star}$, we have
\[
\begin{aligned}
    \mathbb P(\widehat\sigma_R=\sigma \mid B_M) =
    \frac{
        \sum_{t=1}^M(1-p^A_{\rm acc})^{t-1}\nu_A(\sigma)a_A(\sigma)
    }{
        1-(1-p^A_{\rm acc})^M
    }  =
    \frac{\nu_A(\sigma)a_A(\sigma)}{p^A_{\rm acc}}.
\end{aligned}
\]
Substituting the formulas above gives
\[
    \mathbb P(\widehat\sigma_R=\sigma \mid B_M)
    =
    \frac{
        \exp\{\eta \widetilde T_{A,D}(\sigma)\}
    }{
        \sum_{\tau\in\Sigma\cap R_{\sigma^\star}}
        \exp\{\eta \widetilde T_{A,D}(\tau)\}
    },
    \qquad
    \forall \sigma\in\Sigma\cap R_{\sigma^\star}.
\]
The uniform random relabeling maps this canonical Gibbs law to the full
Gibbs law $\pi_A$, exactly as in Lemma \ref{lem:sampler-exact}. Hence, conditional on $B_M$, the output has law $\pi_A$. Now, recall that $R_A$ is the uniform distribution on the orbit
$[\sigma^\star]$, which has cardinality $K!$ since
$\sigma^\star\in\Sigma$ is exact-balanced. Conditional on $B_M^c$, the fallback output is
$\Pi\circ\sigma^\star$ with $\Pi\sim\operatorname{Unif}(\mathfrak S_K)$, whose
law is $R_A$. Thus, we have equality~\eqref{EqnQAM2}.

It remains to prove the two bounds. Since
${\mathsf C}(A)=1$ implies $\kappa\ge \log(4n(K-1))$, we have
\[
    G
    =
    \bigl(1+(K-1)e^{-\kappa}\bigr)^n
    \le
    \left(1+\frac{1}{4n}\right)^n
    \le e^{1/4}.
\]
Since $p^A_{\rm acc}=G^{-1}Z_A$ and $Z_A\ge 1$, we obtain
$p^A_{\rm acc}\ge G^{-1}\ge e^{-1/4}$, so
\[
    r^A_{\rm acc}
    =
    (1-p^A_{\rm acc})^M
    \le
    \bigl(1-e^{-1/4}\bigr)^M
    =
    \overline r_M.
\]

Finally, note that
%since every exact-balanced labeling has orbit size $K!$ and
%$\widetilde T_{A,D}$ is invariant under global relabeling, we have
\[
    \sum_{\tau\in\Sigma}\exp\{\eta\widetilde T_{A,D}(\tau)\}
    =
    K!
    \sum_{\tau\in\Sigma\cap R_{\sigma^\star}}
    \exp\{\eta\widetilde T_{A,D}(\tau)\}.
\]
Thus, for $\sigma\in[\sigma^\star]$, we have $\pi_A(\sigma) = \frac{1}{K!Z_A}$. Since $R_A(\sigma)=\frac{1}{K!}$ on $[\sigma^\star]$, we have
\[
    \frac{R_A(\sigma)}{\pi_A(\sigma)}
    =
    Z_A.
\]
Also, $p^A_{\rm acc}=G^{-1}Z_A\le 1$, so $Z_A\le G\le e^{1/4}$. Therefore,
$R_A(\sigma)\le e^{1/4}\pi_A(\sigma)$ on $[\sigma^\star]$. Outside
$[\sigma^\star]$, we have $R_A(\sigma)=0$, so the same inequality is trivial. This completes the proof.
\end{proof}

If ${\mathsf C}(A)=0$, then $Q_A^{(M)}=\pi_A$, and the claim is immediate. Suppose ${\mathsf C}(A)=1$. By
Lemma~\ref{lem:truncated-certified-decomposition}, we have
\[
    Q_A^{(M)}
    =
    (1-r^A_{\rm acc})\pi_A+r^A_{\rm acc} R_A,
    \qquad
    r^A_{\rm acc}\le \overline r_M,
    \qquad
    R_A\le e^{1/4}\pi_A.
\]
Since $R_A\ge 0$, we have
\[
    Q_A^{(M)}(\sigma)
    \ge
    (1-r^A_{\rm acc})\pi_A(\sigma)
    \ge
    (1-\overline r_M)\pi_A(\sigma).
\]
Also,
\begin{align*}
    Q_A^{(M)}(\sigma) & \le
    (1-r^A_{\rm acc})\pi_A(\sigma)+r^A_{\rm acc} e^{1/4}\pi_A(\sigma) \\
    & =
    \bigl(1+(e^{1/4}-1)r^A_{\rm acc}\bigr)\pi_A(\sigma) \\
    & \le
    \bigl(1+(e^{1/4}-1)\overline r_M\bigr)\pi_A(\sigma).
\end{align*}
The definition of $\xi_M$ gives the two-sided likelihood-ratio bound.
%\end{proof}

%\begin{lemma}[Privacy of the truncated mechanism]
%\label{lemma:truncated-privacy}
%Let $\eta=\varepsilon/(2D)$. Then the full
%truncated mechanism $A\mapsto \widehat\sigma^{(M)}\sim %Q_A^{(M)}$ is pure $(\varepsilon+2\xi_M)$-node-DP.
%\end{lemma}

%\begin{proof}
Now let $A\sim_v A'$ and let $S\subseteq\Sigma$. We have $Q_A^{(M)}(S) \le e^{\xi_M}\pi_A(S)$ and $\pi_{A'}(S) \le e^{\xi_M}Q_{A'}^{(M)}(S)$. Since $\eta=\varepsilon/(2D)$, by Lemma \ref{lem:global-sensitivity}, we have $\pi_A(S)\le e^\varepsilon \pi_{A'}(S)$. Combining the three inequalities proves the claim.

%%%%%

\subsection{Proof of Theorem~\ref{thm:truncated-utility}}
\label{AppTruncUtil}

Fix an input graph $A$ and a true labeling $\sigma_0\in\Sigma$. Since the loss
$r(\sigma_0,\sigma)$ is nonnegative, Lemma~\ref{lem:truncated-uniform-comparison}
gives
\[
\begin{aligned}
    \mathbb E_{Q_A^{(M)}} r(\sigma_0,\sigma) =
    \sum_{\sigma\in\Sigma}r(\sigma_0,\sigma)Q_A^{(M)}(\sigma) \le e^{\xi_M}\sum_{\sigma\in\Sigma}r(\sigma_0,\sigma)\pi_A(\sigma) = e^{\xi_M}\mathbb E_{\pi_A}r(\sigma_0,\sigma).
\end{aligned}
\]
Taking an expectation over $A$, conditioned on $\sigma_0$, yields $\mathbb E_{\sigma_0} r(\sigma_0,\widehat\sigma^{(M)}) \le e^{\xi_M} \mathbb E_{\sigma_0}\mathbb E_{\pi_A}r(\sigma_0,\sigma)$. Taking the supremum over $\sigma_0\in\Sigma$ and applying Theorem \ref{thm:risk} to the Gibbs law $\pi_A$ gives the desired bound.

%%%%%

\subsection{Proof of Corollary~\ref{cor:choice_M_pure}}
\label{AppCorChoice}

First note that since $\overline r_M\le 1/2$, we have $-\log(1-\overline r_M)\le 2\overline r_M$, and
\begin{equation}
\label{remark:xi_M_bound}
\log\bigl(1+(e^{1/4}-1)\overline r_M\bigr)
    \le
    (e^{1/4}-1)\overline r_M
    \le
    2\overline r_M.
\end{equation}
So, $\xi_M \le 2\overline r_M \asymp \overline r_M = e^{-\Theta(M)}$. We look at each claim in turn:

The first claim follows from Lemma \ref{lem:truncated-uniform-comparison} directly.

For the second claim, note that by the lines of argument in the proof of Theorem \ref{thm:runtime}, we have $\Omega_{\rm poly}$, with $\mathbb{P}_{\sigma_0}(\Omega_{\rm poly}) \ge 1 - n^{-c} - (nK)^{-c'A_0}$, and ${\mathsf C}(A) = 1$ on $\Omega_{\rm poly}$. On this event, Algorithm~\ref{AlgOverview} enters the truncated certified branch, and Algorithm \ref{alg:truncated-certified-sampler} is used. Since $M = O(\operatorname{poly}(n))$, the runtime of the algorithm is polynomial-time, by construction.

For the final claim, since we assume $\varepsilon \ge L\log(nK)$, Theorem \ref{thm:truncated-utility} implies
\[
    \sup_{\sigma_0\in\Sigma}
    \mathbb E_{\sigma_0} r(\sigma_0,\widehat\sigma^{(M)})
    \le
    e^{\xi_M}
    \left[
        \exp\left\{-c_1\frac{nI}{K}\right\}
        +
        \frac{1}{nK}e^{-c_2\varepsilon}
        +
        (nK)^{-c_3A_0}
    \right].
\]
Since $\xi_M = e^{-\Theta(M)} \le 1$ and $\frac{nI}{K} = \Omega(\log(nK))$, we have $-c_1\frac{nI}{K} + \xi_M  \le -c_4\frac{nI}{K}$, for some $c_4 > 0$. Since $\varepsilon = \Omega(\log(n))$, we have $-c_2\varepsilon + \xi_M \le -c_5\varepsilon$, for some $c_5 > 0$. Lastly, there exists $c_6 > 0$ such that $(nK)^{-c_3A_0} e^{\xi_M} \le (nK)^{-c_6A_0}$.

%%%%%

\bibliographystyle{plain}
\bibliography{refs}

\end{document}